\documentclass[preprint,12pt]{elsarticle}

\usepackage{amssymb}
\usepackage{amsmath}
\usepackage{amsfonts}
\usepackage{amsthm,mathrsfs}
\usepackage{enumitem}
\usepackage{fancyhdr}
\usepackage{eufrak}

\newtheorem{theorem}{Theorem}[section]
\newtheorem{lemma}[theorem]{Lemma}
\newtheorem{corollary}[theorem]{Corollary}
\newtheorem{proposition}[theorem]{Proposition}

\newtheorem{definition}[theorem]{Definition}
\newtheorem{remark}[theorem]{Remark}

\newcommand{\enurom}[1]
{\begin{enumerate}[label=(\roman*)]
#1
\end{enumerate}}

\newcommand{\enualp}[1]
{\begin{enumerate}[label=(\alph*)]
#1
\end{enumerate}}

\newcommand{\ccd}[0]
{(\cdot)}

\newcommand{\ttnn}[1]
{\textnormal{#1}}

\newcommand{\modulo}[1]
{\left|#1\right|}

\newcommand{\rif}[1]
{(\ref{#1})}

\newcommand{\ssp}[1]
{\hspace{#1mm} }

\newcommand{\norm}[1]
{\left\| #1 \right\|}

\newcommand{\ra}[0]
{\rightarrow}

\newcommand{\ps}[2]
{\langle\,#1,#2\rangle}

\newcommand{\rt}[0]
{(t)}

\newcommand{\rs}[0]
{(s)}

\newcommand{\mineq}[1]
{\leq #1}

\newcommand{\mageq}[1]
{\geq #1}

\newcommand{\equazioneref}[2]
{\begin{equation}\label{#1} \begin{split}
#2
\end{split} \end{equation}}

\newcommand{\eee}[1]
{\begin{equation*} \begin{split}
#1
\end{split} \end{equation*}
}

\newcommand{\graffe}[1]
{\left\{ #1 \right\}}

\newcommand{\tonde}[1]
{\left ( #1 \right )}

\newcommand{\eps}[0]
{\varepsilon }

\newcommand{\ccal}[1]
{\mathcal{#1}}

\newcommand{\scr}[1]
{\mathscr{#1}}

\newcommand{\bb}[1]
{\mathbb{#1}}

\newcommand{\Oonn}[1]
{ \mathrm{#1} }

\newcommand{\Liminf}[0]
{\Oonn{Lim\,inf}}

\newcommand{\Limsup}[0]
{\Oonn{Lim\,sup}}

\newcommand{\halfline}[0]
{ [0,+\infty[ }

\newcommand{\dhaus}[0]
{ {d}_{\scr H} }

\journal{JMAA}

\begin{document}

\begin{frontmatter}

\title{Weak Epigraphical Solutions to Hamilton-Jacobi-Bellman Equations on Infinite Horizon}

\author{Vincenzo Basco}

\affiliation{organization={Thales Alenia Space},
addressline={via Saccomuro 24},
city={Roma},
postcode={00131},
state={Italy},
country={ITA}}

\begin{abstract}
In this paper we show a uniqueness result for weak epigraphical solutions of Hamilton-Jacobi-Bellman (HJB) equations on infinite horizon for a class of lower semicontinuous functions vanishing at infinity. Weak epigraphical solutions of HJB equations, with time-measurable data and fiber-convex, turn out to be viscosity solutions -- in the classical sense -- whenever they are locally Lipschitz continuous. Here we extend the notion of locally absolutely continuous tubes to set-valued maps with continuous epigraph of locally bounded variations. This new notion fits with the lack of uniform lower bound of the Fenchel transform of the Hamiltonian with respect to the fiber. Controllability assumptions are considered.
\end{abstract}


\end{frontmatter}

\section{Introduction}

Since the 80s the investigation of existence and uniqueness of weak solutions to first-order partial differential equations on finite/infinite horizon
was carried out by the meaning of viscosity solutions in the pioneer works of Crandall, Evans, Barles and Lions (\cite{barles1984existenceresults,crandalllionsevans1984someproperties,crandalllions1983viscosity}). Such weak solutions --  also called \textit{viscosity solutions} -- known in the context of control theory, calculus of variations, mean field games, etc... focus on the use of super-sub/solutions. Advances have been made in the direction of HJB equations with autonomous Hamiltonian, by  Souganidis and Ishii  (\cite{ishii1985hbjediscontinuous,ishii1992perron,souganidis1985existence}). The investigation of weak solutions of HJB equations when the Hamiltonian is only time-measurable has become increasingly fundamental due to its applications in applied sciences as in macroeconomy and engineering, although the classical notion of weak solution is unsatisfactory and challenging to manage. In fact, the value function, which is a viscosity solution of the HJB equation, loses its differentiability property -- even in the absence of state constraints -- whenever there are several optimal solutions to the same initial datum or when state constraints are imposed. Avoiding the usage of ``test" functions, by using geometric arguments, the  definition of a weak solution can be equivalently stated in terms of ``normals'' to the epigraph (and hypograph) of the solution, (cfr. \cite{basco2019hamilton,frankowska1993lowersemicontinuous,frankowskaplaskrze1995measviabth} and the references therein), i.e., 
\begin{equation}\label{intro_eq_2}
F(t, x ,- \scr {T}_{{\rm graph}\;u} (t,x,u(t,x))^-)= \{0\}\quad \text{in } ]0,+\infty[\times \Omega
\end{equation}
where $\scr T_E^-(y)$ stands for negative polar of the Boulingad tangent cone of a set $E$ at $y$ and $\Omega\subset \bb R^n$ is an open subset (cfr. Section 2 and Definition \ref{def_epi_sol} below). We point out that, when the dynamics is time-measurable, this equivalence may not be true.

Nevertheless, the study of the uniqueness of weak solutions can be conducted using non-smooth analysis techniques as in the definition reported in [6]. To deal with Hamiltonians measurable in time, Ishii [7] proposed a new notion of weak solution in the class of continuous functions studying the existence and uniqueness of viscosity solutions of HJB equations for the stationary-evolutionary case in finite horizon and on infinite horizon with free state constraints. In the case of Bellman equations, associated with optimal control problems, it is recognized that the viscosity solutions are represented as the corresponding value function. In the general case, the viscosity solutions of the evolutionary equations of HJB with the fiber-convex Hamiltonian, on infinite horizon with state constraints, are represented as the value function of a suitable optimal control problem. In finite horizon, this point of view has been studied by Ishii [8] for the convex case providing a  H\"older continuous representation, and in [9] for Hamiltonians not necessarily convex, but the Lagrangian is simply continuous and the control space is infinite dimensional. On the other hand, in [10] the author constructs a faithful representation, Lipschitz continuous in the state and control. Frankowska et al. [11] studied faithful representations of Hamiltonians measurable in time and their stability, giving clear results on the Lipschitz constants of these representations. In the recent work \cite{basco2020representation}, under weaker hypotheses and assuming the boundedness from above of the Fenchel transform - with respect to the fiber - of the Hamiltonian on its domain, the author has extended the previous representation result, constructing a faithful epigraphical representation (see the references therein \cite{basco2020representation} for further discussions).

To deal with the study of weak solutions in terms of their characterization with the value function of a particular optimal control problem, conditions on the Fenchel transform in the fiber of the Hamiltonian are assumed in the previous works. The uniform lower bound assumption on the Fenchel transform  of the Hamiltonian with respect to the state and the fiber, allows the use of the known viability theorems for tubes investigated by Frankowska et al. in \cite{frankowska1996measurable,frankowskaplaskrze1995measviabth}. Indeed, in this setting  (cfr. also \cite{basco2020representation}), the epigraph of the value function is locally absolutely continuous. However, the assumption of boundedness from below of the Fenchel transform is a not satisfactory condition for many applications, as in mechanics (cfr. \cite{basco2021exploiting}). Indeed, in such a context, the epigraph of the value function of Bellman problems on infinite horizon is at least merely continuous, for Lipschitz problem data.  Therefore, it is needed to extends the notion of locally absolute continuity property of the epigraph, and so further investigations on viability results have to be considered. For this purpose, we give a viability result for continuous tubes of locally bounded variations (see Section 3 and Definition \ref{def_lbv_svm} below).

The central result of this paper is a uniqueness theorem, under controllability assumptions, for weak epigraphical solutions vanishing at infinity (see Theorem \ref{main_theo} below) of HJB equations \rif{intro_eq_2}. By a representation result in \cite{basco2020representation}, each weak epigraphical solution turn out to be an upper and lower weak solution of the value function associated with such representation. Skipping the uniform bound from below of the Fenchel transform, the weak epigraphical solutions are considered merely lower semicontinuous with continuous epigraph of locally bounded variations. 

The outline of this paper is as follows. In Section 2 notations and some known results on non-smooth analysis are collected. In Section 3 we give preliminary definitions with the statement of the main result of this paper. The Section 4 is devoted to the proofs.

\section{Preliminaries and Notations}

 $\bb N$, $|\,.\,|$, and $\ps{.}{.}$ stands for the set of natural numbers, the Euclidean norm, and the scalar product, respectively. Let $E\subset \bb R^k$ be a subset and $x\in \bb R^k$. The closed ball in $\bb R^k$ of radius $r>0$ and centered at $x$ is denoted by $B^k(x,r)$ ($\bb B^k:=B^k(0,1)$ and $\bb S^{k-1}:=\ttnn{bdr } B^k(0,1)$). $\ttnn{cl }E$, $\ttnn{int }E$, $\ttnn{bdr }E$, $E^c$, and $\ttnn{co }E$ stands, respectively, for the closure, the interior, the boundary, the complement, and the convex hull of $E$. The set $E^-=\{p\in \bb R^k\,|\, \ps{p}{e}\mineq 0\, \ttnn{for all } e\in E\}$ is the negative polar of $E$. $\mu$ denotes the Lebesgue measure.

Consider a closed subset $I\subset \bb R$, $C\subset \bb R^k$ non-empty, and $a<b$. We take the following notation:
\enualp{
\item[-] $\ccal L^1(I;C)=\{u:I\ra C \ttnn{ Lebesgue integrable}\}$.
\item[-] $\ccal L^1_{{loc}}(I;C)=\{u:I\ra C\,|\, u\in \ccal L^1(J;C)\; \forall J\subset I \ttnn{ compact}  \}$.
\item[-] $ \ccal L_{{loc}}=\{u\in \ccal L^1_{{loc}}([0,+\infty[;[0,+\infty[) \; |\; \lim_{\eps\ra 0}\, \displaystyle{\sup_{\substack{J\subset [0,+\infty[\\ \mu(J)\mineq \eps}}} \int_{J} {u(s)}\,ds =0\}$ .
\item[-] $\ccal W^{1,1}([a,b];C)=\{u:[a,b]\ra C \ttnn{ absolutely continuous} \}$ endowed with the norm $\norm{u}_{\ccal W^{1,1},[a,b]}:=|u(a)|+\int_a^b|u'(s)|ds$.
\item[-]   $\ccal W^{1,1}_{loc}([a,+\infty[;C)=\{u:[a,+\infty[\ra C\;|\;u\in \ccal W^{1,1}([a,t];C) \;\forall  t\mageq a\}$.
}

Let $A,B\subset \bb R^k$ be two non-empty sets and $x,y\in \bb R^k$. We define the following Euclidean {dist}ances: $d(x,y):=|x-y|$, $\text{dist}(x,B):=\inf\{d(x,b)|b\in B\}$, and $\text{dist}(A,B):=\inf\{d(a,b)|a\in A,b\in B\}$. We use the same notation for the {dist}ance of a point from a set and between two sets -- it does not generate confusion in the context where it is used. We define the so-called \textit{excess} of $A$ beyond $B$ by
$$
exc(A| B):=\sup \{{\text{dist}}(a, B)| {a \in A} \}\in [0,+\infty[\cup \graffe{{+\infty}}.
$$
We recall the following
$
exc(A| B)=\inf \left\{\varepsilon>0 :|A \subset B+\eps\bb B\right\}
$
where $\inf \emptyset=+\infty$, by convention. Moreoveer, the \textit{Pompeiu–Hausdorff \text{dist}ance between A and B} is defined by
\[\dhaus(A,B):=exc(A|B)\vee exc(B|A)\in [0,+\infty[\cup \graffe{{+\infty}}.\]

Let $X,Y$ be two normed spaces and $D\subset X$ be a non-empty closed set. Continuity properties of a set-valued mapping $S: D\subset X \rightsquigarrow  Y$ can be defined on the basis of Painlevé-Kuratowski set convergence. For any $\bar x\in D$ we define, respectively, the upper an lower limit of $S\left( x\right)$ when $ x \rightarrow \bar x$ by
\eee{
\Oonn{Lim\,sup}_{ x \rightarrow \bar x} S\left( x\right)&:= \{y \in Y \mid \liminf _{ x \rightarrow_D \bar x} d(y, S\left( x\right))=0 \}\\
\Oonn{Lim\,inf}_{ x \rightarrow \bar x} S\left( x\right)&:=\{y \in Y \mid \limsup_{ x \rightarrow_D \bar x} d(y, S\left( x\right))=0 \}
}
We say that $S$ is \textit{outer semicontinuous} (osc) at $\bar{x}$ when
$
\Limsup _{x \rightarrow \bar{x}} S(x) \subset S(\bar{x})
$
and \textit{inner semicontinuous} (isc) at $\bar{x}$ when
$
\Liminf _{x \rightarrow \bar{x}} S(x) \supset S(\bar{x}) .
$
It is called Painlevé-Kuratowski \textit{continuous at $\bar{x}\in D$} when it is both osc and isc at $\bar{x}\in D$.
Continuity is taken to refer to Painlevé-Kuratowski continuity, unless otherwise specified. We say that $S$ is continuous if it \textit{continuous} at $\bar x$ for any $\bar x\in D$.

Consider a non-empty subset $E\subset \bb R^k$ and $x\in \ttnn{cl }{E}$. The \textit{Boulingad tangent cone} (or \textit{contingent cone})  and the \textit{Clarke tangent cone} to $E$ at $x$ are defined, respectively, by
\eee{
\scr T_E(x) &:=\Limsup_{t\ra0+}\; t^{-1}(E-x)\\
\scr T_E^C(x) &:=\Liminf_{t\ra0+,\,y\ra_E x}\; t^{-1}(E-y).
}
The \textit{limiting normal cone} and the \textit{regular (or Clarke) normal cone} to $E$ at $x$ are defined, respectively, by
\eee{
\scr N_E(x) &:=\Limsup_{y\ra_E x}{\scr T_E(y)}^-\\ 
\scr N_E^C(x) &:=\scr T_E^C(x)^-.
}
It is known that $\scr T_E^C(x)=\scr N_E(x)^-\subset \scr T_E(x)$ and $\scr T_E^C(x)^- =\ttnn{cl co }\scr N_E(x)$   whenever $E$ is closed (\cite[Chapter 6]{rockafellar2009variational}). 

We denote for any $\eta, r\mageq 0$ the sets
\eee{
N_E(x;\eta)&:=\{ n\in \bb S^{k-1}\,|\, n\in {{\ttnn{cl co}}}\,\scr N_E(y),\,y\in (\ttnn{bdr } E)\cap B^k(x,\eta) \}\\
\Sigma_E(x;\eta,r)&:=\{p\in \bb R^k  \,|\, \forall  n\in N_E(x;\eta),\, \ps{p}{n}\mageq r \}\\
\Gamma_E(x;\eta)&:=\{p\in \bb R^k  \,|\, \exists n\in N_E(x;\eta),\, \ps{p}{n}\mineq 0 \}.
}

Let $\varphi:\bb R^k\ra \bb R\cup \{\pm \infty\}$ be an extended real function. We write $\ttnn{dom }\varphi$, $\ttnn{epi }\varphi$, $\ttnn{hypo }\varphi$, and $\ttnn{graph }\varphi$ for the \textit{domain}, the \textit{epigraph}, the \textit{hypograph}, and the \textit{graph} of $\varphi$, respectively. The Fenchel \textit{transform} (or \textit{conjugate}) of $\varphi$, written $\varphi^*$, is the extended real function $\varphi^*:\bb R^k\ra \bb R\cup \{\pm\infty\}$ defined by
\[\varphi^*(v):=\sup_{p\in \bb R^k}\{\ps{v}{p}-\varphi(p)\}.\]

\section{The Main Result}

Consider a closed non-empty subset $\Omega\subset \bb R^n$.
 We focus our analysis on HJB equation \rif{intro_eq_2} where
\equazioneref{def_F}{
&F(t,x,r,p,q)= r+H(t,x,p,q),\quad r\in \bb R\\
}
and
\equazioneref{def_H}{
H:[0,+\infty[\times \bb R^n\times \bb R^n\times \bb R\ra \bb R
}
is a given Hamiltonian. For any $(t,x,q)\in \halfline\times \bb R^n\times \bb R$, we denote by $H^*(t,x,.,q):\bb R^n\ra \bb R\cup \{\pm \infty\}$ the Fenchel transform of $H(t,x,.,q)$.

\begin{definition}[Weak Epigraphical Solution]\label{def_epi_sol}\rm We say that a lower semicontinuous function $u:\halfline \times \Omega\ra \bb R \cup \{\pm \infty\}$ is an \textit{epigraphical weak solution} of the HJB equation \rif{intro_eq_2} if it satisfies:
\enualp{
\item[] for a.e. $t\mageq 0$ and all $x\in \ttnn{bdr } \Omega$ such that $(t,x)\in \text{dom }u$

$F(t, x ,-\varphi )\mageq 0$ for all $\varphi\in \scr T_{{\rm epi }\; u}(t,x,u(t,x))^- $
}
and
\enualp{

\item[] for a.e. $t\mageq 0$ and all $x\in \text{int } \Omega$ such that $(t,x)\in \text{dom }u$

$F(t, x ,-\varphi)= 0$ for all $\varphi\in \scr T_{{\rm epi }\; u}(t,x,u(t,x))^-$.
}
\end{definition}

\begin{definition}[Set-Valued Maps of Locally Bounded Variations]\label{def_lbv_svm}\rm
Consider a closed interval $I\subset \bb R$. We say that a set-valued map $\Phi:I\rightsquigarrow \bb R^k$ is of \textit{locally bounded variations (LBV)} if it satisfies the following:
\enualp{
\item[-] $\Phi$ takes non-empty closed images;
\item[-] for any $[a,b]\subset I$,
\[
\sup \;\sum_{i=1}^{m-1} exc(\Phi(t_{i+1})\cap \scr K| {\Phi(t_i)})\vee exc(\Phi(t_i)\cap \scr K|{\Phi(t_{i+1})})<+\infty\]
where the supremum is taken over all compact subset $\scr K\subset \bb R^k$ and all finite partition $a=t_1< t_2< ...< t_{m-1}<t_m= b$.
}
\end{definition}

By $\ccal Epi_{loc}([0,+\infty[\times\Omega)$ we denote the family of all lower semicontinuous functions $u:[0,+\infty[\times \Omega\ra \bb R\cup\{\pm \infty\}$ such that
\eee{
(t\mapsto \textnormal{epi }u(t,.)) \ttnn{ is continuous and of LBV on }[0,+\infty[.
}

The following assumptions are considered on $H$:
\enualp{
\item[\textbf{H.1}] for any $t\mageq 0$, $x,p\in \bb R^n$, and $q>0$:
\begin{alignat*}{3}
&(a)\quad H(.,x,p,q) &&\ttnn{ is Lebesgue measurable},\\
&(b)\quad H(t,x,.,q) &&\ttnn{ is convex},\\
&(c)\quad H(t,x,.,.) &&\ttnn{ is positively homogeneus}.
\end{alignat*}

\item[\textbf{H.2}] there exist $\sigma_X,\, \sigma_P,\,\hat\sigma :\halfline\ra \halfline$ measurables, with $\sigma_X\in \ccal L_{\ttnn{loc}}$ and $\hat\sigma$ locally bounded, such that for any $q>0$:
\begin{alignat*}{4}
&\ttnn{for all $t\mageq 0$ and $x_i,p_i,x,p\in \bb R^n$}\\
&\quad (a)\quad |H(t,x_1,p,q)-H(t,x_2,p,q)|\mineq \sigma_X(t)(1+|p|)|x_1-x_2|\\
&\quad(b)\quad|H(t,x,p_1, q)-H(t,x, p_2, q)|\mineq \sigma_P(t)(1+|x|)|p_1- p_2|,\\
&\ttnn{and for a.e. $t\mageq 0$,  all $y\in \bb R^n$, $x\in \ttnn{bdr }\Omega$,   and $p\in \ttnn{dom } H^*(t,x,.,q)$}\\
&\quad(c)\quad{H^*(t,y,p,q)} \mineq \hat \sigma\rt(1+|y|).
\end{alignat*}
}
  
\noindent The next controllability assumptions are also taken into account:

\enualp{
\item[\textnormal{\textbf{C.1}}] there exists $\sigma_{bdr}\in \ccal L_{\ttnn{loc}}$ such that
\begin{alignat*}{3}
\modulo{p}+\modulo{H^*(t,x,p,1)} \mineq \sigma_{bdr}\rt
\end{alignat*}
for a.e. $t\mageq 0$, for all $y\in \bb R^n$, $x\in \ttnn{bdr }\Omega$,   and $p\in \ttnn{dom } H^*(t,x,.,1)$;

\item[\textnormal{\textbf{C.1}}] there exist $ \,\eta>0, \,r>0,\,M\mageq0$ such that\\
$\ttnn{for a.e. } t>0, \forall y\in \ttnn{bdr }\Omega+ \eta \bb B^n,\, \forall p\in \ttnn{dom }H^*(t,y,.,1)\cap \Gamma_\Omega(y;\eta)$,\\ 
$\exists \,w\in\,\ttnn{dom }H^*(t,y,.,1) \cap B^n(p,M)\,:\, w,\,w-p \in \Sigma_\Omega(y;\eta,r)$.
}

\begin{theorem}\label{main_theo}

Consinder the {HJB} equation \ttnn{\rif{intro_eq_2}-\rif{def_H}} with the general assumptions \textnormal{\textbf{H.1-2}}, \textnormal{\textbf{C.1-2}}.

Suppose that for every $(t,x)\in \halfline\times \Omega$

\enualp{
\item $\lim_{T\ra +\infty}\int_{t}^{T} H^*(s,\xi\rs,\xi'\rs,1)\, ds$ exists for any $\xi\in \ccal W^{1,1}_{loc}([t,+\infty[;\Omega)$ such that $\xi(t)=x$,

\item the infimum $\alpha(t,x)$ defined by
\eee{
\inf \; \Big\{\lim_{T\ra +\infty}\int_{t}^{T} H^*(s,\xi\rs,\xi'\rs,1) \,ds \;\Big|  \xi\in \ccal W^{1,1}_{loc}([t,+\infty[;\Omega),\xi(t)=x   \Big\}
}
is finite.
}
Then there exists only one weak epigraphical solution $u\in \ccal Epi_{loc}([0,+\infty[;\bb R^{n+1})$ of the HJB equation with vanishing condition:
\equazioneref{vanishing_cond}{
F(t, x , - \scr {T}_{\rm{graph }\;u} (t,x,u(t,x))^-)&= \{0\}\quad \text{in } ]0,+\infty[\times \Omega\\
\lim_{t\ra +\infty}\; \sup_{x\in \Omega}{|u(t,x)|}&=0.
}
 
\end{theorem}
\begin{remark}\rm\
\enualp{
\item In the the previous result no regularity assumptions are made w.r.t. the $q$-dependence on the Hamiltonian $H$, beyond those which are made in \textbf{H.1}-(a),(b). Moreover, condition \textbf{H.1}-(c) does not ensures non-negative values for $H^*(t,x,.,q)$. 

\item We underline that, from assumption \textbf{H.2}-(c), for all $t\in I$, $x \in \bb {R}^{n}$, $q>0$, and $\bar q\in \ttnn{cl }\ttnn{dom } H^*(t,x,.,q)$ there exists $ \eps>0$ such that
\[\sup \;\{H^*(t,x,p,q)| p\in \ttnn{dom } H^*(t,x,.,q)\cap B^n(\bar q,\eps)\}<+\infty.\]
}
\end{remark}

\section{Proofs}

In this section we provide the proofs of the results of this paper. We recall that, in the literature, a map $E:[0,+\infty[\rightsquigarrow  \bb{R}^{d}$ is also called \textit{tube}.

\subsection{Viability for Continuous Tubes of Locally Bounded Variations}

The following result is well known (cfr. \cite[Chapter 3]{dontchev2009implicit}).

\begin{lemma}[Characterization of Continuity of Set-Valued Maps, \cite{dontchev2009implicit}]\label{lemma_cont} Let $D\subset \bb R^n$ be a closed non-empty set. Consider a set-valued map $S: D\subset \bb R^n \rightsquigarrow  \bb R^m$ and let $\bar{x} \in D\cap \Oonn{dom } S$.

Then $S$ is continuous at $\bar{x}$ if and only if the function ${\text{dist}} (x, S(.))$ is continuous at $\bar{x}$ in $D$ for every $x \in \bb R^m$.

\end{lemma}

Lemma below provide a generalization of Lemma 4.8 in \cite{frankowskaplaskrze1995measviabth} to continuous tubes of LBV.

\begin{lemma}\label{lemma1base}

Assume that $E:[0,+\infty[\rightsquigarrow  \bb{R}^{d}$ is continuous of locally bounded variations and the set-valued map $Y:[0,+\infty[\rightsquigarrow  \bb R^d$ satisfies
\enualp{
\item $ Y(.)$ is measurable for every $x \in \bb {R}^{d}$ with non-empty compact images;

\item there exists $\rho\in \ccal L^1_{{loc}}([0,+\infty[;[0,+\infty[)$ such that $d_{\scr H}{(Y\rt,Y\rs)} \mineq \int_s^t \rho(h)dh$ for all $0\mineq s\mineq t$.
}

Let $\Psi$ be defined by
\[t \mapsto \Psi(t):={\textnormal{dist}}\left(E(t), Y(t)\right).\]

Then $\Psi$ is continuous and of locally bounded variations on $[0,+\infty[$.

\end{lemma}

\begin{proof}
We first show that the function
\equazioneref{claim1}{
(t,x)\mapsto {\text{dist}}(x, E(t)) \ttnn{ is continuous in }[0,+\infty[\times \bb R^d.
}
Indeed, from the continuity of $E$ and Lemma \ref{lemma_cont} we have that the function $t\mapsto \text{dist}(0,E\rt)$ is continuous on $[0,+\infty[$. Furthermore, for any $x\in \bb R^n$, it follows that the map $t\mapsto \text{dist}(x,E\rt)$ is continuous on $[0,+\infty[$ because, by applying the same argoment above, $t\rightsquigarrow  E\rt-\graffe{x}$ is continuous as well. Hence, since for any $t,s\in [0,+\infty[$ and $x,y\in \bb R^d$ the triangular inequality yield
\eee{
|\text{dist}(x,E\rt)-\text{dist}(y,E\rs)|&\mineq |\text{dist}(x,E\rt)-\text{dist}(x,E\rs)|+|x-y|,
}
it follows \rif{claim1}.

Now, fix $[a,b]\subset [0,+\infty[$. Notice that, by the triangular inequality and our assumptions, for every $a\mineq s\mineq b$
\eee{
\dhaus(\{0\},Y\rs)&\mineq \dhaus(\{0\},Y(b))+\dhaus(Y\rs,Y(b))\\
&\mineq \dhaus(\{0\},Y(b))+\int_a^b\rho(h)dh=:r,
}
so we deduce that
$
Y\rs\subset r\bb B^d$ for all $s\in [a,b].
$
Using \rif{claim1} and the Weierstrass theorem, we put $R:=r+\max\{{\text{dist}}(x, E(t)): x\in r\bb B^d,t\in [a,b]\}$. Hence, it follows for any $s\in [a,b]$
\[\text{dist}(E\rs\cap R\bb B^d,Y\rs)=\text{dist}(E\rs,Y\rs).\]
Since for any $s,t\in [0,+\infty[$ and any compact $\scr K \subset \bb R^d$
\[
E\left(s\right) \cap \scr K \subset E\left(t\right)+exc\left(E\left(s\right)\cap \scr K| E\left(t\right)\right) \bb B^d,
\]
keeping $\scr K=R\bb B^d$, we get
\eee{
&\Oonn{\text{dist}}\left(Y\left(s\right), E\left(s\right)\right)\\
& \mineq \Oonn{\text{dist}}\left(Y\left(s\right), E\left(t\right)\right)+exc\left(E\left(s\right)\cap \scr K| E\left(t\right)\right) \vee exc\left(E\left(t\right)\cap \scr K| E\left(s\right)\right)
}
and
\eee{
&\Oonn{\text{dist}}\left(Y\left(s\right), E\left(t\right)\right)\\
& \mineq \Oonn{\text{dist}}\left(Y\left(s\right), E\left(s\right)\right)+exc\left(E\left(s\right)\cap \scr K| E\left(t\right)\right) \vee exc\left(E\left(t\right)\cap \scr K| E\left(s\right)\right) .
}
Thus, for every $s,t\in [0,+\infty[$
\eee{
&\left|\Oonn{\text{dist}}\left(Y\left(s\right), E\left(t\right)\right)-\Oonn{\text{dist}}\left(Y\left(s\right), E\left(s\right)\right)\right| \\
&\mineq exc\left(E\left(s\right)\cap \scr K| E\left(t\right)\right) \vee exc\left(E\left(t\right)\cap \scr K| E\left(s\right)\right).
}
Finally, there exists $M>0$, depending only on $[a,b]$, such that, for any partition $a=t_1< t_2< ...< t_{m-1}<t_m= b$,
\equazioneref{eq:1}{
&\sum_{i=1}^{m-1}\left|\Psi \left(t_{i+1}\right)-\Psi\left(t_{i}\right)\right| \\
&\leq \sum_{i=1}^{m-1}\left| {\text{dist}}\left(Y\left(t_{i+1}\right), E\left(t_{i+1}\right)\right)- {\text{dist}}\left(Y\left(t_{i+1}\right), E\left(t_{i}\right)\right)\right|\\
&\qquad +\sum_{i=1}^{m-1} d_{\scr H}\left(Y\left(t_{i+1}\right), Y\left(t_{i}\right)\right)\\
&\mineq \sum_{i=1}^{m-1} exc\left(E\left(t_{i+1}\right)\cap \scr K| E\left(t_i\right)\right) \vee exc\left(E\left(t_i\right)\cap \scr K| E\left(t_{i+1}\right)\right)\\
&\qquad +\int_a^b\rho(h)dh\\
&\mineq M.
}
Hence, the locally bounded variations property for real valued functions follows.

Next we show that $\Psi$ is uniformly continuous in $[a,b]$. Fix $\eps>0$. For any $\tau\in [a,b]$ consider $y_\eps(\tau)\in Y(\tau)$ such that
\[\text{dist}(Y(\tau),E(\tau))+\frac{\eps}{4} \mageq \text{dist} (y_\eps(\tau),E(\tau)).\]
Then, for any $s,t\in [a,b]$ and any $x_s\in Y\rs$
\equazioneref{ineq_2}{
&\Psi\rs-\Psi\rt \\
&\mineq \text{dist}(Y\rs,E\rs)-\text{dist}(y_\eps\rt,E\rs)+\frac{\eps}{4}\\
&\mineq \text{dist}(x\rs,E\rs)-\text{dist}(y_\eps\rt,E\rt)+\frac{\eps}{4}\\
&\mineq |x_s-y_\eps\rt|+|\text{dist}(y_\eps\rt,E\rs)-\text{dist}(y_\eps\rt,E\rt)|+\frac{\eps}{4}.
}
We notice that, from assumption (b), there exists $\delta>0$, depending only on $[a,b]$, such that for every $s,t\in [a,b]$ with $|s-t|\mineq \delta$
\equazioneref{unif_cont_Y}{
\dhaus(Y\rs,Y\rt) \mineq \frac{\eps}{4}.
}
Furthermore, applying the triangle inequality and the Lipschitz continuity of the \text{dist}ance function, for any $s,t\in [a,b]$
\equazioneref{ineq_3}{
&|\text{dist}(y_\eps\rt,E\rs)-\text{dist}(y_\eps\rt,E\rt)|\\
&\mineq |\text{dist}(y_\eps\rs,E\rs)-\text{dist}(y_\eps\rt,E\rs)|\\
&\quad +|\text{dist}(y_\eps\rs,E\rs)-\text{dist}(y_\eps\rt,E\rt)|\\
&\mineq |y_\eps\rs-y_\eps\rt|+|\text{dist}(y_\eps\rs,E\rs)-\text{dist}(y_\eps\rt,E\rt)|.
}
We recall that, from \rif{claim1}, the map $(\tau,x)\mapsto \text{dist}(x,E(\tau))$ is uniformly continuous on $[a,b]\times R\bb B^d$, with $R>0$ depending on $[a,b]$ as above. Hence, replacing $\delta$ with a sufficiently small one and using \rif{unif_cont_Y}, for any $s,t\in [a,b]$ with $|s-t|\mineq \delta$ holds
\equazioneref{ineq_4}{
|\text{dist}(y_\eps\rs,E\rs)-\text{dist}(y_\eps\rt,E\rt)|\mineq \frac{\eps}{4}.
}
Moreover, from assumption (a) and applying the Measurable Selection Theorem (\cite[Theorem 8.1.3]{aubin2009set}), we can find a measurable selection
$x(\tau)\in Y(\tau)$ for all $ \tau\in[0,+\infty[.$ Thus, keeping $x_s=x\rs$ in \rif{ineq_2}, using \rif{unif_cont_Y}, \rif{ineq_3}, and \rif{ineq_4}, we conclude that
\eee{
&\Psi\rs-\Psi\rt\mineq \eps\quad \forall s,t\in [a,b]\ttnn{ such that }|s-t|\mineq \delta.
}
From the symmetry with respect to $s$ and $t$ in the previous inequality, the conclusion follows.
\end{proof}

In what follows, we denote by
\[
D^+\Psi(t):=\limsup_{h\ra0+} \frac{\Psi(t+h)-\Psi(t)}{h}\in \bb R\cup\{\pm \infty\}
\]
the right Dini derivative at $t\in \bb R$ of a real valued function $\Psi\ccd$. Before to state the main result of this section, we need the following

\begin{lemma} \label{lemma_gronw}
Let $\Psi :[\tau,T]\ra \bb R$ be a continuous function and $\alpha, \beta :[\tau,T]\ra \bb R$ be two locally bounded functions, with $\alpha(\cdot) \geq 0$, such that
$$
D^+\Psi(t) \leq \alpha(t) \Psi(t)+\beta(t) \quad \text { for all } t \in ] \tau, T[.
$$
Then, for every $t \in[\tau, T[$,
$$
\Psi(t) \leq \Psi(\tau) e^{\alpha (t-\tau)}+\int_{\tau}^{t} e^{\alpha(t-r)} \beta \,d r
$$
where $\alpha:=\sup_{s\in [\tau,T]} \alpha\rs$ and $\beta:=\sup_{s\in [\tau,T]} |\beta\rs|$.
\end{lemma}
\begin{proof}
Let $\delta>0$ and define $
\varphi_{\delta}(t):=(\Psi(\tau)+\delta) e^{\alpha (t-\tau)}+\int_{\tau}^{t} e^{\alpha(t-r)}(\beta +\delta) d r$. Then, $\varphi_{\delta}^{\prime}(t)=\alpha \varphi_{\delta}(t)+\beta +\delta$ in $\left[\tau, T\left[\right.\right.$ and $\varphi_{\delta}(t)>\Psi(t)$ for all $t \in[\tau, T[$ close to $\tau$. We show that $\varphi_{\delta}(\cdot) \geq \Psi(\cdot)$ for any $\delta>0$. By contraddiction, assume that there are some $\left.\left.\bar t \in\right] \tau, T\right]$ and $\delta>0$ with $\varphi_{\delta}\left(\bar t\right)<\Psi\left(\bar t\right)$. Setting
$
s:=\inf \left\{t \in\left[\tau, \bar t\right] \mid \varphi_{\delta}(t)<\Psi(t)\right\},
$
we obtain that $\varphi_{\delta}\left(s\right)=\Psi\left(s\right) $ and $ \tau<s<\bar t $.
Thus, from the definition of $s$,
$$
\begin{aligned}
\varphi_{\delta}^{\prime}\left(s\right)=\liminf _{h\ra 0+} \frac{\varphi_{\delta}\left(s+h\right)-\varphi_{\delta}\left(s\right)}{h} & \leq \limsup _{h\ra 0+} \frac{\Psi\left(s+h\right)-\Psi\left(s\right)}{h} \\
& \leq \alpha\left(s\right) \Psi\left(s\right) +\beta\left(s\right) \\
& \leq \alpha \varphi_{\delta}\left(s\right) +\beta,
\end{aligned}
$$
i.e., $\alpha \varphi_{\delta}\left(s\right)+\beta +\delta \leq \alpha \varphi_{\delta}\left(s\right) +\beta$. Then a contradiction follows.

\end{proof}

Below we give a relaxation of the result in \cite{frankowskaplaskrze1995measviabth}, in which the stronger condition of the absolutely continuity of the tube $E$ is considered. The proof is a mild adaptation, based in turn on \cite{jarnik1977conditions}.

\begin{proposition}\label{prop_viability}
Assume that $E:[0,+\infty[\rightsquigarrow  \bb{R}^{d}$ is continuous and of locally bounded variations in sense of Definition \ref{def_lbv_svm}, let $t_0\in [0,+\infty[$, $x_0\in E(t_0)$, and $\Phi:[0,+\infty[\times \bb R^d\rightsquigarrow \bb R^d$ be a set-valued map with non-empty convex closed values such that
\eee{
(i)&\quad \Phi(.,x) \ttnn{ is measurable for any $x $};\\
(ii)&\quad \exists \rho\in \ccal L^1_{loc}([0,+\infty[;[0,+\infty[)\\
&\quad \forall r>0\,\exists \,k_r:\halfline \ra \halfline \ttnn{ locally bounded: for a.e. } t\\
&\qquad \qquad (a)\quad \sup_{v\in \Phi (t,x)}|v|\mineq \rho(t)(1+|{x}|),\\
&\qquad \qquad (b)\quad \Phi(t,.) \textnormal{ is $k_r(t)$-Lipschitz on } r\bb B^d.
}
Then, if for a.e. $t>0$ and all $y\in E(t)$
\[{\ttnn{cl co }} \scr T_{\text {graph } E}(t, y)\cap (\{1\} \times \Phi(t,y))\neq \emptyset,\]
for any $T>t_0$ and $t_0\mineq t_i<t_{i+1} \mineq T$, $i=1,...,m$, there exists on $[t_0,T]$ a solution $x\ccd$ of
\equazioneref{x_primo_in_phi}{
x'\rt\in \Phi(t,x\rt) \text{ a.e. }t 
}
satisfying $x(t_0)=x_0$ and
\[ x\left(t_{i}\right) \in E\left(t_{i}\right)\text{ for all } i=1,...,m+1.\]
\end{proposition}

\begin{proof} 
First of all, we notice that, by the Gronwall's Lemma and assumption (ii)-(a), for all $r>0$ there exists $R>r$ such that if an absolutely continuous function $x:[0, t_{1}] \ra \bb {R}^{d}$ satisfies $|x^{\prime}(t)| \mineq \rho(t)(1+|x(t)|)$ a.e. in $\left[0, t_{1}\right]$, $x(0)=x_{0}$, and $|x_0|\mineq r$, then $|x(t)|\mineq R$ for all $t \in\left[0, t_{1}\right]$. Moreover, observe that ${\Phi}$ is integrably bounded on $[0,t_1]$, i.e. for almost all $t \in [0,t_1]$ and all $x \in r \bb B^d$, $|v| \leq \rho(t)(1+R):={\rho_R}(t)$ for any $v\in {\Phi}(t, x)$.

Fix $t_0\in [0,+\infty[$ and $x_0\in E(t_0)$. Consider the map $\Psi$ of the Lemma \ref{lemma1base} applyied with $Y\ccd$ defined by the reacheable set
\eee{
Y(s)&=R[t_0,x_0](s):=\{x(s)\,|\, x\ccd \text{ solution of } \rif{x_primo_in_phi}, \, x(t_0)=x_0\}.
}
We first claim that
\equazioneref{phi_constant_zero}{
\Psi(t)=0 \quad \forall t\in]t_0,+\infty[,
}
arguing by contradiction. Suppose that $T>t_{0}$ with $\Psi(T)>0$ and consider $\tau=\sup \left\{t<T \mid \Psi(t)=0\right\}$. So, $\Psi>0$ on $\left.] \tau, T\right]$ and $\Psi\left(\tau\right)=0$. We divide the proof of the claim \rif{phi_constant_zero} into two steps.

\textsc{Step 1 (Estimate on differentiability points):}
From Lemma \ref{lemma1base}, the well known Scorza-Dragoni property (cfr. e.g. \cite[Chapter 2]{filippov2013differential},\cite{scorza1952applicazione}), and the Mean Value Theorem for set-valued maps (cfr. \cite{aubin2012differential}), there exists a subset $\scr C \subset[0, T]$ of full measure such that for all $t \in \scr C$ and $x \in \bb {R}^{d}$ the following three properties hold: $\Psi$ is differentiable at $t$; for every $v \in \Phi(t, x)$ there exists on $[t,T]$ a solution to the problem
$
y^{\prime} \in \Phi(t, y), \; y(t)=x, \; y^{\prime}(t)=v
$; for every $y\ccd$ solution of \rif{x_primo_in_phi} on $[t_0,t]$ satisfying $y(t)=x$ and for every sequence $h_{i} \rightarrow 0+$ we have
$
\emptyset \neq \Oonn{Limsup}_{i \rightarrow \infty}\{ \frac{y\left(t-h_{i}\right)-x}{h_{i}} \}\subset-\Phi(t, x)
$. Consider $t\in \scr C$ and let $z \in Y(t)$, $y \in E(t)$ satisfy $\Psi(t)=|z-y| $ and put $
p=\frac{z-y}{|z-y |}
$.  We first prove that for all $(u, w) \in \scr T_{\text {graph } E}(t, y)$
\equazioneref{step_1}{
\text{$\Lambda(u,w)=\Phi(t,z)$ if $u\mageq 0$ and $\Lambda(u,w)\neq \emptyset$ if $u<0$}
}
where we denoted $\Lambda(u,w):=\{v\in \Phi(t,z):\Psi'(t)u\mineq \langle p, u v-w\rangle \}$.
Indeed, let $(u, w) \in \scr T_{\text {graph } E }(t, y)$ and $h_{i} \rightarrow 0+, u_{i} \rightarrow u, w_{i} \rightarrow w$ satisfying $y+h_{i} w_{i} \in E\left(t+h_{i} u_{i}\right)$ for all $i\in \bb N$. Suppose that there exists a subsequence $\graffe{u_{i_k}}_{k\in \bb N}$ with $u_{i_{k}} \mageq 0$ for all $k\in \bb N$. Let $v \in \Phi(t, z)$ and $x \ccd$ a solution of \rif{x_primo_in_phi} on $[t,T]$ such that $x(t)=z$ and $x^{\prime}(t)=v$. Thus
$$
\Psi\left(t+h_{i_{k}} u_{i_{k}}\right)-\Psi(t) \mineq |x\left(t+h_{i_{k}} u_{i_{k}}\right)-y-h_{i_{k}} w_{i_{k}} |-|z-y|.
$$
Dividing by $h_{i_{k}}$ and taking the limit we get $\Psi^{\prime}(t) u \mineq \langle p, u v-w\rangle$. Otherwise, we have $u_{i}<0$ for all $i$ large enough. In this case, consider a solution $\bar{x} \ccd$ of \rif{x_primo_in_phi} on $[t_0,t]$, $\{i_{k}\}_{k\in \bb N}$, and $\bar{v} \in \Phi(t, z)$ such that $\bar x(t_0)=x_0,\,\bar{x}(t)=z$, and
$
\lim _{k \rightarrow \infty} \frac{\bar{x} (t+h_{i_{k}} u_{i_{k}} )-z}{ h_{i_{k}}}=u \bar{v}
$.
Hence for all $k\in \bb N$ sufficiently large
$$
\Psi\left(t+h_{i_{k}} u_{i_{k}}\right)-\Psi(t) \mineq |\bar{x}\left(t+h_{i_{k}} u_{i_{k}}\right)-y-h_{i_{k}} w_{i_{k}} |-|z-y|.
$$
Dividing by $h_{i_{k}}$ and taking the limit we get $\Psi^{\prime}(t) u \mineq \langle p, u \bar{v}-w\rangle$. Hence, it follows \rif{step_1}. Now, consider $e_{j} \mageq 0$ and $\left(u_{j}, w_{j}\right) \in \scr T_{\text {graph }E}(t, y)$ for $j=0, \ldots, d$ such that
$
\sum_{j=0}^{d} e_{j}=1$ and $u:=\sum_{j=0}^{d} e_{j} u_{j}>0$. Without loss of generality, we may assume that for some natural number $0 \mineq N < d$ and all $j=1,...,N$ we have $u_{j} \mageq 0$ and $u_j<0$ for all $j=N+1,...,d$. From \rif{step_1}, for every $j=N+1,...,d$ there exists $\bar{v}_{j} \in \Phi(t, z)$ such that
$
\Psi^{\prime}(t) u_{j} \mineq \left\langle p, u_{j} \bar{v}_{j}-w_{j}\right\rangle.
$
Thus, applying again \rif{step_1} it follows that
\equazioneref{convessita_phi_somme}{
\Psi^{\prime}(t) (\sum_{j=0}^{N} e_{j} u_{j} ) &\mineq \langle p, \sum_{j=0}^{N} e_{j} u_{j} v-\sum_{j=0}^{N} e_{j} w_{j} \rangle,\quad \forall v \in \Phi(t, z),\\
\Psi^{\prime}(t) (\sum_{j=N+1}^d e_{j} u_{j} ) &\mineq \langle p, \sum_{j=N+1}^d e_{j} u_{j} \bar{v}_{j}-\sum_{j=N+1}^d e_{j} w_{j} \rangle.
}
Notice that, since
$e_{N+1}\left|u_{N+1}\right|+...+e_{d}\left|u_{d}\right|=\left| e_{N+1} u_{N+1}+...+ e_{d} u_{d}\right|< e_{0} u_{0}+...+ e_{N} u_{N}$, we have
\[\theta_j:=\frac{e_{j}\left|u_{j}\right|}{\sum_{j=0}^{N} e_{j} u_{j}}\Rightarrow \sum_{j=N+1}^d \theta_j\in ]0,1[ \]
that, due to convexity of $\Phi(t, z)$, it implyes in turn that
$
(1-\sum_{j=N+1}^d \theta_j) v+\sum_{j=N+1}^d \theta_{j} \bar{v}_{j} \in \Phi(t, z)
$ for every $ v \in \Phi(t, z).$ Hence, recalling \rif{step_1} and \rif{convessita_phi_somme}, we obtain for all $v \in \Phi(t, z)$
\eee{
\Psi^{\prime}(t) u = \Psi^{\prime}(t) (\sum_{j=0}^{d} e_{j} u_{j} )&\mineq \langle p, (\sum_{j=0}^{N} e_{j} u_{j}-\sum_{j=N+1}^d e_{j} |u_{j} | ) v\\
&\quad\qquad +\sum_{j=N+1}^d e_{j} ( |u_{j} |+u_{j} ) \bar{v}_{j}-\sum_{j=0}^{d} e_{j} w_{j} \rangle\\
&= \langle p, u v-\sum_{j=0}^{d} e_{j} w_{j} \rangle.
}
So, we have that
\eee{
& \forall t\in \scr C,\forall \,(u, w) \in {\ttnn{cl co }} \scr T_{\text {graph } E}(t, y),\, \text{with } u>0 \\
&\exists h_{i} \rightarrow 0+,\exists u_{i} \rightarrow u,\exists w_{i} \rightarrow w:\\
& y+h_{i} w_{i} \in E\left(t+h_{i} u_{i}\right) \ttnn{ for all } i\in \bb N\ttnn{ and}\\
& \lim_{i\ra +\infty} \frac{\Psi(t+h_iu_i)-\Psi(t)}{h_i}=\Psi^{\prime}(t) u \mineq \left\langle p, u v-w\right\rangle,\; \forall v \in \Phi(t, z)
}
$\ttnn{where }z \in Y(t), y \in E(t), \ttnn{ satisfy } \Psi(t)=|z-y|, \ttnn{ and } p:=\frac{z-y}{|z-y |}.$

\textsc{Step 2 (Upper estimate of the Right Dini Derivative):}
Fix $\eps>0$ and $t\in]\tau,T[$. Keep a sequence of positive numbers $h_k\ra0+$ such that $\lim_{k\ra +\infty} \frac{\Psi(t+h_k)-\Psi(t)}{h_k}=D^+\Psi(t)$. From Lemma \ref{lemma1base}, for all $k\in \bb N$ there exists $\delta_k>0$ such that for any sequences $\graffe{s_k}_k,\graffe{\tilde s_k}_k\subset [\tau,T]$
\equazioneref{passo_a_zero}{
|s_k-\tilde s_k|\mineq \delta_k \quad \forall k\; \Longrightarrow\; |\Psi(s_k)-\Psi(\tilde s_k)|\mineq o(h_k)\quad \forall k.
}
Moreover, applying again Lemma \ref{lemma1base}, we can find a sequence of differentiability points $\graffe{t_k}_{k\in \bb N}\subset \scr C$ for $\Psi$ such that
\equazioneref{stima_punti_differ}{
\ttnn{$\eps h_k+\Psi(t)\mageq \Psi(t_k)$ and $|t_k-(t+h_k)|\mineq \frac{\delta_k}{3}$ for all $k\in \bb N$}.
}
For any $k\in \bb N$ consider $z_k \in Y(t_k)$ and $y_k \in E(t_k)$ such that $\Psi(t_k)=|z_k-y_k| $ and put $p_k=\frac{z_k-y_k}{|z_k-y_k|}$. From Step 1, for any $k\in \bb N$ and any $(u_k,w_k)\in \text{cl co }\scr  T_{\ttnn{graph }E}(t_k,y_k)$, with $u_k>0$, there exist $h^{(k)}_j\rightarrow 0+, u^{(k)}_j \rightarrow u_k$, and $w^{(k)}_{j} \rightarrow w_k$ satisfying
\[y_k+h^{(k)}_j w^{(k)}_{j} \in E (t_k+h^{(k)}_j u^{(k)}_j )\quad \forall j\in \bb N\]
and
\[\lim_{j\ra +\infty}\frac{\Psi(t_k+ h^{(k)}_ju^{(k)}_j)-\Psi(t_k)}{ h^{(k)}_j}\mineq \ps{p_k}{u_k v-w_k} \quad \forall v\in \Phi(t_k,z_k).\]
In particular, it follows that for any $k\in \bb N$ we can choose $j_k\in \bb N$ satisfying
\equazioneref{passo_b_0}{
|h^{(k)}_j|\mineq h_k,\quad |h^{(k)}_ju^{(k)}_j|\mineq \frac{\delta_k \wedge h_k}{3},\quad \forall j\mageq j_k.
}
Hence, from \rif{passo_a_zero} and  \rif{stima_punti_differ}, we have for any large $k\in \bb N$ and any $j\mageq j_k$
\eee{
\frac{\Psi(t+h_k)-\Psi(t_k+h^{(k)}_ju^{(k)}_j)}{h_k}=o(h_k)
}
and
\equazioneref{passo_c}{
&\frac{\Psi(t+h_k)-\Psi(t)}{h_k}\\
&\mineq \frac{\Psi(t+h_k)-\Psi(t_k+h^{(k)}_ju^{(k)}_j)}{h_k}+\frac{\Psi(t_k+ h^{(k)}_ju^{(k)}_j)-\Psi(t_k)}{h_k}+ \eps.
}
From \rif{passo_b_0} and Step 1, we get for every large $k\in \bb N$
\[
\limsup_{j\ra +\infty} \frac{\Psi(t_k+ h^{(k)}_ju^{(k)}_j)-\Psi(t_k)}{ h_k}\mineq \ps{p_k}{u_kv-w_k},\quad \forall v\in \Phi(t_k,z_k).
\]
Using that, inequality in \rif{stima_punti_differ}, assumption (ii)-(b), and since for all $k\in \bb N$ we can find $(1,v_k)\in \text{cl co } \scr T_{\ttnn{graph } E}(t_k,y_k)$ with $v_k\in \Phi(t_k,y_k)$,
passing in \rif{passo_c} to the upper limit as $j\ra \infty$ we get
\equazioneref{passaggio_limite_per_granw}{
\frac{\Psi(t+h_k)-\Psi(t)}{h_k}&\mineq o(h_k)+k_r(t_k) \Psi(t_k)+\eps\\
&\mineq o(h_k)( 1+k_r\eps) +\eps + k_r\Psi\rt
\quad \ttnn{for all large $k\in \bb N$}
}
where $k_r=\sup_{t\in[\tau,T]}|k_r\rt|$. Passing first to the limit in \rif{passaggio_limite_per_granw} as $k\ra+\infty$, and then using the arbitrariness of $\eps$, we get
\[D^+\Psi\rt\mineq k_r \Psi\rt.\]
From that and Lemma \ref{lemma_gronw}, the claim \rif{phi_constant_zero} follows immediately.

To conclude the proof, consider $t_0\mineq t_i<t_{i+1} \mineq T$, $i=1,...,m$. Following the induction argument, assume that for some $j \mageq 0$ there exists an absolutely continuous trajectory $y:[t_0,t_i]\ra \bb R^d$ solving \rif{x_primo_in_phi} such that $y\left(t_{i}\right) \in E\left(t_{i}\right)$ for all $i \mineq j$. From the above claim applied with $\left(t_{0}, x_{0}\right)$ replaced by $\left(t_{j}, y\left(t_{j}\right)\right)$ we can find an absolutely continuous trajectory $\bar{y}:[t_j,T]\ra \bb R^d$ solving \rif{x_primo_in_phi} such that $\bar y(t_j)=y(t_j)$ and $\bar{y}\left(t_{j+1}\right) \in E\left(t_{j+1}\right)$. Thus we can extend $y$ on the time interval $\left[t_{j}, t_{j+1}\right]$ by setting $y(t)=\bar{y}(t)$ for all $t \in\left[t_{j}, t_{j+1}\right],$ and the proof is now complete.
\end{proof}

\begin{corollary}[Viability for Continuous of LBV Tubes]\label{corollario_viabilita} If all the assumptions of Proposition \ref{prop_viability} hold, then for any $t_0\in [0,+\infty[$ and $x_0\in E(t_0)$ there exists an absolutely continuous viable solution
\eee{
&x'\rt\in \Phi(t,x\rt)\; \text{ for a.e. }t>t_0\\
&x(t_0)=x_0\\
&x(t) \in E(t)\quad \forall t>t_0.
}
\end{corollary}

\begin{proof}
We take the same notations as in the proof of Proposition \ref{prop_viability}. Let $t_0\in [0,+\infty[$, $x_0\in E(t_0)$, $T>t_0$,
and $\varepsilon>0$. Pick $\delta>0$ such that the map $\Psi$ in the Lemma \ref{lemma1base} is uniformly continuous in $[t_0,T]$. By replacing $\delta$ with a suitable small one, we can assume that $\int_s^{\bar s}\rho(t)dt\mineq \eps$ for any $s,\bar s\in [t_0,T]$ such that $|s-\bar s|\mineq \delta$. Let a finite partition $t_{0}<t_{1}<\cdots<t_{m}=T$ be such that $t_{i+1}-t_{i}<\delta$.
Then, from the choice of $\delta$, the Gronwall's Lemma, and Proposition \ref{prop_viability}, there exist a constant $c>0$ (depending only on $|x_0|$ and $T$) and a trajectory $x_\eps\ccd$ solving \rif{x_primo_in_phi} on $[t_0,T]$ such that for any $t \in\left[t_{0}, T\right]$ and $i$ with $t_{i} \mineq t\mineq t_{i+1}$
\eee{
\text{dist}(x_\eps(t),E\rt)&\mineq \text{dist}(Y_i\rt,E\rt)+\text{dist}(x_\eps(t),Y_i\rt) \\
&= \text{dist}(Y_i\rt,E\rt)-\text{dist}(Y_i(t_{i+1}),E(t_{i+1}))\\
&\qquad +\text{dist}(x_\eps(t),Y_i\rt)\\
&\mineq \eps+c \int_{t_i}^t \rho(s)ds\\
&\mineq (1+c)\eps
}
where $Y_i(s)=R[t_i,x(t_i)](s)$. Now, applying again the Gronwall's Lemma, consider a sequence $\{x_{\eps_{j}}\ccd\}_{j\in \bb N}$ converging weakly to some absolutely continuous function $x:[t_0,T]\ra \bb R^d$, where $\varepsilon_{j} \rightarrow 0+$. It follows that $x$ solves \rif{x_primo_in_phi}, with $x(t_0)=x_0$, and $x(t) \in E(t)$ for all $t \in\left[t_{0}, T\right]$. Using an iterative argument, we can extend such solution to the whole halfline $[t_0,+\infty[$ in order to get the statement.
\end{proof}

\subsection{Proof of Theorem \ref{main_theo}}

Next, we recall a Representation result of time-measurable and fiber-convex Hamiltonians recently investigated in (\cite[Proposition 4.1]{basco2020representation}).

\begin{proposition}[Representation, \cite{basco2020representation}]\label{theo_rep_H}
Assume \textnormal{\textbf{H.1}-(a),(b)}  and \textnormal{\textbf{H.2}}. Then, there exists an epigraphical representation
\[ ({\mathfrak{F}_\#},{\mathfrak{L}_\#} ): I \times \bb R^n \times \bb R\times \bb R^{n+1} \ra \bb {R}^{n}\times \bb R\]
with $({\mathfrak{F}_\#},{\mathfrak{L}_\#})(., x,q, \theta)$ measurable for any $x\in \bb R^n$, $q>0$, $\theta\in \bb R^{n+1}$ and satisfying:
\enurom{
\item[\ttnn{(i)}] for any $t\in I$, $x,p\in \bb R^n$, and $q>0$
\eee{
H(t,x,p,q)&=\sup\,\{\ps{(p,-1)}{({\mathfrak{F}_\#}(t,x,q,\theta),{\mathfrak{L}_\#}(t,x,q,\theta))} \: |\;\theta\in \bb B^{n+1}\};
}
\item[\ttnn{(ii)}]
$
|({\mathfrak{F}_\#},{\mathfrak{L}_\#})(t, x_1,q, \theta_1)-({\mathfrak{F}_\#},{\mathfrak{L}_\#})(t, x_2,q, \theta_2))|\mineq C(t,x_1,x_2,\theta_1,\theta_2,q)
$
for any $t\in I$, $x_1,x_2\in \bb R^n$, and $\theta_1,\theta_2\in \bb R^{n+1}$, where

$
C(t,x_1,x_2,\theta_1,\theta_2,q):=5(n+1)( \sigma_X(t) |x_1-x_2|+ |\eta(t,x_1)\theta_1-\eta(t,x_2){\theta_2}|)$

$\eta(t,x):=\sigma_P(t)(1+|x|)+\gamma(t,x)+|H(t,x,0,q)|$

$\gamma(t,x):=0 \vee \sup\; \{ H^*(t,x,p,q) \;|\; p\in \ttnn{dom } H^*(t,x,.,q)\};$
\item[\ttnn{(iii)}] for any $t\in I$, $x\in \bb R^n$, and $q>0$
\begin{alignat*}{3}
&(a)&\quad \ttnn{dom } H^*(t,x,.,q)&={\mathfrak{F}_\#}(t,x,q,\bb B^{n+1})\\
&(b)&\quad \ttnn{graph }H^*(t,x,.,q)&\subset ({\mathfrak{F}_\#},{\mathfrak{L}_\#})(t,x,q,\bb B^{n+1})\\
&(c)&\quad \ttnn{epi }H^*(t,x,.,q)&= ({\mathfrak{F}_\#},{\mathfrak{L}_\#})(t,x,q,\bb R^{n+1}).
\end{alignat*}
}

Moreover, if in addition \textnormal{\textbf{H.1-(c)}} holds, then we have the following representation
\eee{
H(t,x,p,q)&=\sup\,\{\ps{(p,-q)}{(f(t,x,\theta),{\mathfrak{L}}(t,x,\theta))}\;|\;\theta\in \bb B^{n+1}\}\\
&=\sup\,\{\ps{(p,-q)}{(f(t,x,\theta),\scr L(t,x,\theta))}\;|\;\theta\in \bb B^{n+1}\}
}
where
\[
f(t,x,\theta):=\mathfrak{ F }_{\#}(t,x,1,\theta) \quad \& \quad {\mathfrak{L}}(t,x,\theta):=\mathfrak{L }_\#(t,x,1,\theta)
\]
\[\scr L(t,x,\theta):=H^*(t,x,f(t,x,\theta),1).\]
%
\end{proposition}
\begin{proof}
Let $({\mathfrak{F}_\#},{\mathfrak{L}_\#})$ be the representation given by \cite[Theorem 4.1]{basco2020representation} and satisfying the statements (i)-(iii). Assuming further \textnormal{\textbf{H.1}}-(c), from (i) we have for any $q >0$
\eee{
H(t,x,p,q)&=qH(t,x,\frac{p}{q},1)\\
&=q \sup\,\{\ps{\frac{p}{q}}{{\mathfrak{F}_\#}(t,x,1,\theta)}-{\mathfrak{L}_\#}(t,x,1,\theta)|\theta\in \bb B^{n+1}\}.
}
Thus, from (iii) and the proof of \cite[Theorem 4.1]{basco2020representation}, we get
\eee{
H(t,x,p,q)&=\sup\,\{\ps{{p}}{{\mathfrak{F}_\#}(t,x,1,\theta)}-q{\mathfrak{L}_\#}(t,x,1,\theta)|\theta\in \bb B^{n+1}\}\\
&=\sup\,\{\ps{(p,-q)}{(f(t,x,\theta), H^*(t,x,f(t,x,\theta),1))}\;|\;\theta\in \bb B^{n+1}\}.
}
\end{proof}

In what follows, we consider the representation associated with the Hamiltonian $H$
\[(f,{\mathfrak{L}}):[0,+\infty[\times \bb R^n\times \bb B^{n+1}\ra \bb R^n\times \bb R\]
provided by Proposition \ref{theo_rep_H}.
For any $(t,x)\in [0,+\infty[\times \Omega$, we denote by $\ccal U_\Omega(t,x)$ the -- possibly empty -- set of all pairs $(\xi,\theta):[t,+\infty[\ra \bb R^n\times \bb R^{n+1}$ such
\equazioneref{sistemacontrollo}{
&\xi'\rs=f(s,\xi\rs,\theta\rs),\quad \theta\rs\in \bb B^{n+1} \quad\text{for a.e. }s\in[t,+\infty[\\
&\xi(t)=x\\
&\xi\ccd\subset \Omega.
}
The \textit{value function} $\scr V_{f,{\mathfrak{L}}}:[0,+\infty[\times \Omega \ra \bb R\cup \{\pm \infty\}$ associated with the representation $(f,{\mathfrak{L}})$ is defined by\footnote{If $u\in \ccal L^1_{{loc}}([a,+\infty[;\bb R)$, we denote by $\int_{a}^{\infty}u\rs\,ds:=\lim_{b\ra+\infty}\int_{a}^{b}u\rs\,ds $, provided this limit exists.}
\begin{eqnarray*}
\scr V_{f,{\mathfrak{L}}}(t,x):=
\inf \;\Big\{ \lim_{T\ra+\infty} \;\int_t^{T} {\mathfrak{L}}(s,\xi\rs,\theta\rs)\,ds \;|\;  (\xi,\theta)\in \ccal U_\Omega(t,x)\Big\}
\end{eqnarray*}
where $\inf \emptyset=+\infty$ by convention.

Now, consider two epigraphical weak solutions of the HJB equation, namely $v$ and $w$, satisfying the vanishing condition in \rif{vanishing_cond}. It is sufficient to show that $v\mineq \scr V_{f,{\mathfrak{L}}}\mineq w$. Fix $(t_0,x_0)\in ]0,+\infty[\times \Omega$. By our assumptions, there exists $T>t_0$ such that
\equazioneref{cond_W_finale}{
\modulo{w(t,y)}\mineq\eps,\;\modulo{v(t,y)},\mineq\eps \quad \forall \,t\mageq T,\; \forall y\in \Omega.
}
We divide the proof of Theorem \ref{main_theo} into parts \textbf{(A)}, \textbf{(B)}, and \textbf{(C)}.

{\textbf{(A)}}: We have
\equazioneref{step1}{
w(t_0,x_0)\mageq \scr V_{f,{\mathfrak{L}}}(t_0,x_0).
}
If $w(t_0,x_0)=+\infty$, then $w(t_0,x_0)\mageq \scr V_{f,{\mathfrak{L}}}(t_0,x_0)$. So, assume that $(t_0,x_0)\in \ttnn{dom}\,w$. In order to prove \rif{step1}, we show the following
\equazioneref{stima_W_per_passaggio_al_limite}{
&\text{$\exists (\xi,\theta):[t_0,+\infty[\ra \bb R^n\times \bb R^{n+1}$ solving \rif{sistemacontrollo} :}\\
&w(t_0,x_0)\mageq w(t,\xi\rt)+\int_{t_0}^{t}\scr L(s,\xi\rs, \theta\rs)\,ds\quad \forall\,t\mageq t_0.
}
We postpone the proof of \rif{stima_W_per_passaggio_al_limite} and we assume temporarily it is valid.
Using the vanishing condition and passing to the upper limit in \rif{stima_W_per_passaggio_al_limite} as $t\ra \infty$ yields,  for every $(\xi\ccd,\theta\ccd)\in \scr U_\Omega(t_0,x_0)$,
\[
w(t_0,x_0)\mageq \limsup_{t\ra+\infty} \int_{t_0}^{t} \scr L(s,\xi\rs,\theta\rs)\,ds.
\]
In particular, it follows that
\begin{gather}\label{eq_2_1}
    \begin{aligned}
       w(t_0,x_0)&\mageq \inf \; \Big\{\limsup_{t\ra +\infty}\int_{t_0}^{t} H^*(s,\xi\rs,\xi'\rs,1) \,ds \;|  
        \;\begin{aligned}
            & \xi\in \ccal W^{1,1}_{loc}([t_0,+\infty[;\Omega),\\
            & \xi(t_0)=x_0
        \end{aligned}
        \;    \Big\} \\
&\,\quad =\alpha(t_0,x_0).
    \end{aligned}
\end{gather}
%
%
By our assumptions, $\alpha(t_0,x_0)>-\infty$. Fix $\eps>0$ and consider a trajectory
$\xi\in \ccal W^{1,1}_{loc}([t_0,+\infty[;\bb R^n)$ with $\xi(t_0)=x_0$ and $ \xi\ccd\subset \Omega$ satisfying
$$\int_{t_0}^{+\infty}   H^*(s,\xi\rs,\xi'\rs,1) ds <\alpha(t_0,x_0)+\eps.$$
We have that
$(\xi'\rs,z'\rs)\in \ttnn{graph }   H^*(s,\xi\rs,.,1)$ for a.e. $s\mageq t_0$,
where we put $z\rs:=\int_t^{s}  H^*(\tau,\xi(\tau),\xi'(\tau),1)\,d \tau$ for all $s\mageq t_0$. Applying now Proposition \ref{theo_rep_H}-(iii)-(b) and the Measurable Selection Theorem, we have that there exists a measurable function $\theta:[t_0,+\infty[\ra \bb B^{n+1}$ such that $(\xi'\rs,z'\rs)=(f(s,\xi\rs,\theta\rs),{\mathfrak{L}}(s,\xi\rs,\theta\rs))$ for a.e. $s\mageq t_0$. So, for all $  t\mageq t_0$
\eee{
& \int_{t_0}^{t} H^*(s,\xi\rs,\xi'\rs,1) \,ds\\
&=\int_{t_0}^{t} z'\rs\,ds =\int_{t_0}^{t} {\mathfrak{L}}(s,\xi\rs,\theta\rs)\,ds,
}
that imply
\eee{
\int_{t_0}^{+\infty}  H^*(s,\xi\rs,\xi'\rs,1) \,ds
\mageq \scr V_{f,{\mathfrak{L}}}(t_0,x_0).
}
We get $\alpha(t_0,x_0)+\eps > \scr V_{f,{\mathfrak{L}}}(t_0,x_0)$. Since $\eps$ is arbitrary, we have
$
\alpha(t_0,x_0)\mageq \scr V_{f,{\mathfrak{L}}}(t_0,x_0).
$
Recalling \rif{eq_2_1}, it follows the inequality in \rif{step1}.

To conclude the proof of \textbf{(A)}, we have only to show \rif{stima_W_per_passaggio_al_limite}. 
Since $w$ is a weak epigraphical solution of the HJB, applying the representation result Proposition \ref{theo_rep_H} there exists a set $C\subset [0,+\infty[$ with $\mu(C)=0$ such that for all $(t,x)\in \ttnn{dom}\,w\cap (([0,+\infty[ \backslash C)\times \Omega)$
\equazioneref{eq_1_1}{
&F(t, x ,\varphi=(r,p,q) )\\
&=-r+\sup\graffe{\ps{f(t,x,\theta)}{-p} +q{\mathfrak{L}}(t,x,\theta) :\theta\in \bb B^{n+1} }\mageq 0\\
& \forall\,(r,p,q)\in {\scr T_{\ttnn{epi }\, w}(t,x,w(t,x))}^-.
}
Hence, from \rif{eq_1_1}, the representation result Proposition \ref{theo_rep_H}, and the Separation Theorem, we deduce that
\equazioneref{con_viabilita_G_intersezione}{
\tonde{\graffe{1}\times \hat \Phi(t,x)}\cap {\ttnn{cl co}}\,\scr T_{\ttnn{epi}\,w }(t,x,w(t,x))\neq \emptyset
}
for all $(t,x)\in \ttnn{dom}\,w\cap (([0,+\infty[ \backslash C)\times \Omega)$ where
\begin{gather*}
    \begin{aligned}
       \hat \Phi(t,x):= \{(&f(t,x,\theta),-\scr L(t,x,\theta)-r) \;\Big|
        \;\begin{aligned}
             &\theta\in \bb B^{n+1},\\
             &r\in [0,\hat \sigma\rt(1+\modulo{x})-\scr L(t,x,\theta)]
        \end{aligned}
        \;    \Big\}.
    \end{aligned}
\end{gather*}
Putting, for all $(t,x,z)\in \bb R\times \bb R^n \times \bb R$,
\[ \Phi(t,x,z):= \hat \Phi(t,x),\]
from assumptions \textbf{H.1-2} and applying Corollary \ref{corollario_viabilita} with $E\rt=\ttnn{epi}\,w(t,\cdot)$, there exists an absolutely continuous trajectory $X_0\ccd= (\xi_0\ccd, z_0\ccd)$ solving
\equazioneref{X_epi_W}{
&X'\rt\in \Phi(t,X\rt)\quad\ttnn{for a.e. }t\in [t_0,t_0+1]\\
&\xi(t_0)=x_0,\,z(t_0)=w(t_0,x_0)\\
&\xi(t)\in \Omega, \, z\rt\mageq w(t, \xi\rt)\quad\forall t\in [t_0,t_0+1].
}
We claim that for any poisitive $j\in \bb N$ the trajectory $X_0$ admits an extension on the interval $[t_0,t_0+j]$ to a trajectory $X_j$ satisfying \rif{X_epi_W} on $[t_0,t_0+j]$. We proceed by the induction argument on $j\in \bb N$. Let $j\in \bb N$ and suppose that $X_j\ccd=(\xi_j\ccd,z_j\ccd)$ satisfies the claim. Then, using \rif{con_viabilita_G_intersezione} and applying again Corollary \ref{corollario_viabilita} on the time interval $[t_0+j,t_0+j+1]$, we can find a trajectory $X\ccd=(\xi\ccd,z\ccd)$ satisfying
\eee{
&X'\rt\in \Phi(t,X\rt)\quad\ttnn{for a.e. }t\in [t_0+j,t_0+j+1]\\
&\xi(t_0+j)=\xi_j(t_0+j),\,z(t_0+j)=z_j(t_0+j)\\
&\xi(t)\in \Omega,\, z\rt\mageq w(t, \xi\rt) \quad\forall t\in [t_0+j,t_0+j+1].
}
Putting $X_{j+1}\rt=(\xi_j\rt,z_j\rt)$ if $t\in [t_0,t_0+j]$ and $X_{j+1}\rt=(\xi\rt,z\rt)$ if $t\in ]t_0+j,t_0+j+1]$, we deduce that $X_{j+1}\ccd$ satisfies our claim. Now, consider the trajectory $X\rt=(\xi\rt,z\rt)$ given by
$X\rt=X_j(t)$ if $t\in [t_0+j,t_0+j+1].$ By the Measurable Selection Theorem, there exist two measurable functions $\theta\ccd$ and $r\ccd$, with $\theta\rt \in \bb B^{n+1}$ and $r\rt\in[0,\hat \sigma\rt(1+\modulo{\xi\rt})-\scr L(t,\xi\rt,\theta\rt)]$ for a.e. $t\mageq t_0$, such that
\[z\rt=w(t_0,x_0)-\int_{t_0}^{t}\scr L(t,\xi\rs, \theta\rs)\,ds-\int_{t_0}^{t} r(s)\,ds\mageq w(t,\xi(t))\]
for all $t\mageq t_0$. Hence,
inequality in \rif{stima_W_per_passaggio_al_limite} immediately follows.

\textbf{(B)}:
We show
\equazioneref{step2}{
v(t_0,x_0)\mineq \scr V_{f,{\mathfrak{L}}}(t_0,x_0).
}
If $\scr V_{f,{\mathfrak{L}}}(t_0,x_0)=+\infty$, then $\scr V(t_0,x_0)\mageq v(t_0,x_0)$. So, let us assume that $(t_0,x_0)\in \ttnn{dom}\,\scr V_{f,{\mathfrak{L}}}$. Fix $\eps>0$. Let $(\bar \xi\ccd,\bar \theta\ccd)$ be an optimal trajectory-control pair at $(t_0,x_0)$ for $\scr V_{f,{\mathfrak{L}}}$ and consider $s_i\ra +\infty$ with $\graffe{s_i}_{i\in \bb N}\subset ]T,+\infty[$.
Put $\bar X\ccd=(\bar \xi\ccd,\bar z\ccd)$
where $\bar z\rt=-\int_{t_0}^{t}\scr L(t,\bar\xi\rs,\bar\theta\rs) \,ds$.
From assumptions \textbf{H.1-2} and  \textbf{C.1-2}, applying the Neighboring Feasible Trajectory result Proposition \ref{teo_neig} in Appendix, we deduce that for any $i\in \bb N$ there exists a trajectory $X_i\ccd=(\xi_i\ccd,z_i\ccd)$ solving
\eee{
&X'_i\rt\in \Phi(t,X_i\rt)\quad \ttnn{for a.e. }t\in [t_0, s_i]\\
&X_i(s_i)=(\bar \xi(s_i),\bar z(s_i))\\
&\xi_i(t)\in \ttnn{int}\,\Omega\quad \forall\, t\in [t_0, s_i[
}
and
$$\lim_{i\ra \infty}\;\sup\{ |{X_i\rs-\bar X\rs}|\;|\; s\in [t_0, s_i]\}=0.$$
Hence, by the Measurable Selection Theorem, for any $i\in \bb N$ there exists a measurable selection $\theta_i(t) \in \bb B^{n+1}$ such that $(\xi_i\ccd,\theta_i\ccd)$ satisfies
\eee{
&\xi'_i\rt=f(t,\xi_i\rt,\theta_i\rt) \quad\text{for a.e. }t\in[t_0, s_i]\\
&\xi_i(s_i)=\bar \xi(s_i)\\
&\xi_i\rt \in \ttnn{int}\,\Omega \quad\forall \,t\in[t_0, s_i[
}
and
\equazioneref{lim_x_i_t_0}{
\lim_{i\ra \infty}\xi_i(t_0)= \bar \xi(t_0),
}
\equazioneref{lim_k_int_L}{
\lim_{i\ra \infty}{{\int_{t_0}^{ s_i}
\scr L(t,\bar\xi_i\rs,\bar\theta_i\rs) \,ds=\int_{t_0}^{\infty}\scr L(t,\bar\xi\rs,\bar\theta\rs) \,ds}}.
}
Now, fix $i\in \bb N$ and consider $\graffe{\tau_j}_j\subset ]T,s_i[$ with $\tau_j\ra s_i$. Note that, by the dynamic programming principle, $\xi_i(\tau_j)\in \ttnn{dom}\, \scr V_{f,{\mathfrak{L}}}(\tau_j,\cdot)$ for all $j\in \bb N$. We need the following

\begin{lemma}\label{lemma_parte_interna_2}
For any $0<\tau_0<\tau_1$ and any pair $(\xi,\theta)$ solution of
\equazioneref{lemma_feas_traj_eq}{
&\xi'\rs=f(s,\xi\rs,\theta\rs), \quad \theta\rs\in \bb B^{n+1} \quad \textnormal{for a.e. }s\in [\tau_0,\tau_1] \\
&\xi([\tau_0,\tau_1])\subset \ttnn{int}\,\Omega\\
&(\tau_0,\xi(\tau_0))\in \ttnn{dom}\,v,
}
we have
\eee{
( \xi\rt,v(\tau_0,\xi(\tau_0))-\int_{\tau_0}^{t} \scr L(s, \xi\rs, \theta\rs)ds)\in \ttnn{epi}\,v(t,.)\quad \forall t\in [\tau_0,\tau_1].
}

%
\end{lemma}
\begin{proof}
Since $v$ is an epigraphical solution and from Proposition \ref{theo_rep_H}, there exists a set $C\subset [0,+\infty[$ with $\mu(C)=0$ such that for all $(t,x)\in \ttnn{dom}\,v\cap (([0,+\infty[ \backslash C)\times \ttnn{int }\Omega)$
\eee{
&F(t, x ,r,p,q )\\
&=-r+\sup\graffe{\ps{f(t,x,\theta)}{-p} +q\scr L(t,x,\theta) :\theta\in \bb B^{n+1} }= 0\\
& \forall\,(r,p,q)\in {\scr T_{\ttnn{epi}\, v}(t,x,v(t,x))}^-.
}
Notice that, by the separation theorem, this is equivalent to
\eee{
\graffe{-1}\times (-\hat \Phi(t,x))\subset \ttnn{cl }{\ttnn{co}}\,\scr T_{\ttnn{epi}\,v}(t,x,y)
}
for all $y\mageq v(t,x)$ and all $(t,x)\in ((]0,\infty[\backslash C)\times \ttnn{int}\,\Omega)\cap \ttnn{dom}\,v$. Let $0< \tau_0<\tau_1$. Thus
\equazioneref{1_f_L}{
(1,f^{\circ}(t,x,\theta),{\scr L}^{\circ}(t,x,\theta))\in {\ttnn{cl co}}\,\scr T_{\ttnn{graph}\,E}(t,x,y)
}
for a.e. $t\in[0,\tau_1-\tau_0]$, any $(x,y)\in E(t)\cap (\ttnn{int}\,\Omega\times \bb R)$, and any $\theta\in \bb B^{n+1}$, where $f^{\circ}(t,x,\theta):=-f(\tau_1-t,x,\theta)$, ${\scr L}^{\circ}(t,x,u):={\scr L}(\tau_1-t,x,\theta)$, and $E(t):=\ttnn{epi}\,v(\tau_1-t,\cdot)$. Consider a trajectory-control pair $( \xi\ccd, \theta\ccd)$ solving \rif{lemma_feas_traj_eq}, with $ \xi([\tau_0,\tau_1])\subset \ttnn{int}\,\Omega$ and $(\tau_1,\xi(\tau_1))\in \ttnn{dom}\,v$. Denote by $z$ the solution of
\eee{
&z'\rt=-{\scr L}(t, \xi\rt, \theta\rt)\quad \ttnn{for a.e. }t\in [\tau_0,\tau_1]\\
&z(\tau_1)=v(\tau_1, \xi(\tau_1)).
}
Put $ \theta^\circ\ccd=\theta(\tau_1-\cdot)$, and $\xi^\circ\ccd=\xi(\tau_1-\cdot)$, $ z^\circ\ccd=z(\tau_1-\cdot)$ the unique solutions of ${\xi^\circ}'\rt=f^{\circ}(t,\xi^\circ\rt, \theta^\circ\rt)$, $ {z^\circ }'\rt={\scr L}^{\circ}(t,\xi^\circ\rt, \theta^\circ\rt)$ a.e. $t\in [0,\tau_1-\tau_0]$, respectively, with $\xi^\circ(0)=\xi(\tau_1)$ and $ z^\circ(0)=v(\tau_1,\xi(\tau_1))$. 
Applying Corollary \ref{corollario_viabilita} with $\Phi$ given by the single-valued map 
\[(t,x)\rightsquigarrow \graffe{(f^{\circ}(t,\xi^\circ\rt, \theta^\circ\rt),{\scr L}^{\circ}(t,\xi^\circ\rt, \theta^\circ\rt))},\]
it follows the conclusion.
\end{proof}
\noindent We continue the proof of part $\textbf{(B)}$. Consider the solution $g_j\ccd$ of the Cauchy problem
\eee{
&g'\rt=-\scr L(t,\xi_i\rt,\theta_i\rt) \quad\ttnn{for a.e. }t\in [t_0,\tau_j]\\
&g(\tau_j)=v(\tau_j,\xi_i(\tau_j)).
}
From Lemma \ref{lemma_parte_interna_2} above and since $\mathfrak{L}\mageq \scr L$, we have that
\eee{
\int_{t_0}^{\tau_j}{\mathfrak{L}}(s,\xi_i\rs,\theta_i\rs)\,ds+v(\tau_j,\xi_i(\tau_j))\mageq v(t_0,\xi_i(t_0))\quad \forall j\in \bb N.
}
Hence, by \rif{cond_W_finale},
\eee{
\int_{t_0}^{\tau_j}{\mathfrak{L}}(s,\xi_i\rs,\theta_i\rs)\,ds+\eps\mageq v(t_0,\xi_i(t_0))\quad \forall j\in \bb N,
}
and taking the limit as $j\ra \infty$ we get $\int_{t_0}^{s_i}{\mathfrak{L}}(s,\xi_i\rs,\theta_i\rs)\,ds+\eps\mageq v(t_0,\xi_i(t_0))$. Passing now to the lower limit as $i\ra \infty$, using \rif{lim_x_i_t_0}, \rif{lim_k_int_L}, and the lower semicontinuity of $v$, we have $\int_{t_0}^{\infty}{\mathfrak{L}}(s,\bar \xi\rs,\bar \theta\rs)\,ds+\eps\mageq v(t_0,x_0)$, i.e.,
$\scr V_{f,{\mathfrak{L}}}(t_0,x_0)+\eps\mageq v(t_0,x_0).$
Since $\eps$ is arbitrary, we conclude
$$\scr V_{f,{\mathfrak{L}}}(t_0,x_0)\mageq v(t_0,x_0).$$

\textbf{(C)}: From parts \textbf{(A)} and \textbf{(B)} we have that $v\mineq \scr V_{f,{\mathfrak{L}}} \mineq w \ttnn{ on } ]0,+\infty[\times \Omega$,
that in turn imply $v=w$ on $]0,+\infty[\times \Omega$.
Finally, since $t\rightsquigarrow \ttnn{epi}\,w(t,\cdot)$ is continuous, $w$ is lower semicontinuous, and applying Lemma \ref{lemma_cont}, we have $\liminf_{s\ra 0+,\, y\ra_\Omega x}w(s,y)=w(0,x)$ for all $x\in \Omega$. So, for any $x_0\in \Omega$
\eee{
w(0,x_0)=\liminf_{s\ra 0+,\, y\ra_\Omega x_0}w(s,y)=\liminf_{s\ra 0+,\, y\ra_\Omega x_0}v(s,y)=v(0,x_0).
}

\section*{Appendix}

We provide here a proof of a neighboring feasible trajectory result (used in the proof of Theorem \ref{main_theo}) involving uniform linear estimates on intervals $I=[t_0,t_1]$, with $0\leqslant t_0<t_1$, for state constrained differential inclusion of the form
\equazioneref{sist_app}{
&x' (t) \in Q(t,x (t) )\quad {\rm a.e. } \;\;t\in I\\
&x (t) \in \ccal A\quad \forall\, t\in I
}
where $Q:\ [0,+\infty[ \times \mathbb{R}^n \rightsquigarrow  \mathbb{R}^n$ is a given set-valued map and $\ccal A\subset \mathbb{R}^n$ is a closed non-empty set. A function $x:[t_0,t_1] \to \mathbb{R}^n$ is said to be an \textit{$Q$-trajectory} if it is absolutely continuous and $x' (t) \in Q(t,x (t) )$ for a.e. $t\in [t_0,t_1]$, and a \textit{feasible $Q$-trajectory} if $x (\cdot) $ is an $Q$-trajectory and $x([t_0,t_1]) \subset \ccal A$.

We consider the following assumptions on $Q(\cdot,\cdot)$:
\begin{enumerate}[label=(\roman*)]
\item[] \textbf{Ax.1} $Q$ has closed, non-empty values, a sub-linear growth, and $Q(\cdot,x)$ is Lebesgue measurable for all $x\in \mathbb{R}^n$;

\item[] \textbf{Ax.2} \label{lipF} there exists $\varphi\in \mathcal L_{{\rm loc}}$ such that $Q(t,\cdot)$ is $\varphi (t) $-Lipschitz continuous for a.e. $t\in\halfline$.

\end{enumerate}

We consider also the following controllability assumptions:

\begin{itemize}

\item[] \textbf{Ax.C.1} $\exists \, q\in \mathcal L_{{\rm loc}} \; \mbox{\rm such that } \; Q(t,x)\subset q(t)\,\mathbb{B}^n $ for all $ x\in \ttnn{bdr }\ccal A, $ and a.e. $ t\in \ [0,+\infty[ $;

\item[]  \textbf{Ax.C.2} there exist $ \,\eta>0, \,r>0,\,M\mageq0$ such that\\
for a.e. $t\in \ [0,+\infty[ $, all $y\in \ttnn{bdr }\ccal A + \eta \mathbb{B}^n$ and $v\in Q(t,y)\cap (- \Gamma_{\ccal A}(y;\eta))$\\
there exists $w\in\,Q(t,y)\cap B^n(v,M)$ : $w,\,w-p \in (-\Sigma_{\ccal A}(y;\eta,r))$.

\end{itemize}

\begin{proposition}[Neighboring Feasible Trajectory]\label{teo_neig}
Let us assume \ttnn{\textbf{{Ax.1-2, Ax.C.1-2}}}. Then for all $\delta>0$ there exists a constant $\beta>0$ such that for any $[t_0,t_1]\subset [0,+\infty[$, with $t_1-t_0=\delta$, any $Q$-trajectory $\hat x\ccd$ on $[t_0,t_1]$ with $\hat x(t_0)\in \ccal A$, and any $\rho>0$ with
\eee{
\rho\mageq dist(\hat x\rt, \ccal A) \quad \forall t\in [t_0,t_1],
}
there exists a $Q$-trajectory $x\ccd$ on $[t_0,t_1]$, with $x(t_0)=\hat x(t_0)$, satisfying
\eee{
x\rt \in \ttnn{int}\,\ccal A\ssp{5}\forall t\in ]t_0,t_1]
}
\eee{
|{\hat x\rs-x\rs}| \mineq \beta \rho\quad \forall s\in [t_0,t_1].
}
\end{proposition}

\begin{proof} The proof is inspired by \cite{bascofrankowska2018lipschitz} (see also the references therein for further Neighboring Feasible Trajectory results). We denote by $\theta_g(.)$ the modulus of continuity of a real function $g$.
Fix $\delta>0$ and let $k>0,\,\Delta>0$ be such that $k>1/r$, and
\equazioneref{cond2_1}{
(i)\;8\Delta M\mineq \eta;\; (ii)\; 2e^{\theta_\varphi(\Delta)}{\theta_\varphi(\Delta)M}<r;\;(iii)\;2e^{\theta_\varphi(\Delta)}{\theta_\varphi(\Delta)M}k<rk-1.
}
We notice that, from \ttnn{\textbf{{Ax.2}}} and \ttnn{\textbf{{Ax.C.1}}}, for any $\alpha>0$ there exists $q_\alpha\in \mathcal L_{{\rm loc}}$ such that $Q(t,x)\subset q_\alpha\rt \bb B^n$ for a.e. $t>0$ and all $x\in \text{bdr } \ccal A + \alpha \bb B^n$. So, by the proof of \cite[Theorem 1]{bascofrankowska2018lipschitz}, without loss of generality we can replace $M$ with a suitable greater constant and suppose that $\delta\mineq \Delta$, $\int_J q_{\eta}  \mineq M$ for any $J\subset [0,+\infty)$ with $\mu(J)=\delta$. Moreover, by \rif{cond2_1}-(i) we have that if $\hat x(t_0)\in \ccal A \backslash(\ttnn{bdr }\ccal A+\dfrac{\eta}{4}\bb B^n)$ then the conclusion follows taking $x\ccd=\hat x\ccd$. Then we suppose that $\hat x(t_0)\in (\ttnn{bdr }\ccal A+\dfrac{\eta}{4}\bb B^n)\cap \ccal A$.
Let us denote the non-smooth orientated \text{dist}ance $Dist_\ccal A\ccd$ from the boundary of $\ccal A$ by\footnote{We refer to \cite{clarke1990optimization,vinter00} for the well known properties about such {dist}ance.}
\[Dist_\ccal A(.):=dist(.,\ccal A)-dist(.,{\ccal A^c})\]
and define the measurable set
\begin{gather*}
    \begin{aligned}
       \ccal A_+=\Big\{s\in [t_0,t_1]\cap J\,\Big|
        \;\;\begin{aligned}
             &\exists x\in (\ttnn{bdr }\ccal A)+B^n(\hat x(s),\eta),\\
             &\exists n\in \scr N^C_\ccal A(x)\cap \bb S^{n-1},\\
             & \ps{n}{\hat x'(s)}\mageq 0
        \end{aligned}
        \;\;    \Big\}.
    \end{aligned}
\end{gather*}
where $J=\graffe{s\in[t_0,t_1]\,|\,\hat x'(s)\ttnn{ exists}}$. We show the stated uniform estimate by distinguish the cases $\mu(\ccal A_+)=0$ or $\mu(\ccal A_+)>0$.

\textsc{Case 1}: $\mu(\ccal A_+)=0$.

\noindent Let $t\in(t_0,t_1]$. Then applying the Main Value Theorem (cfr. \cite{clarke1990optimization}), for some $z\rt\in[\hat x(t_0),\hat x(t)]$ and $\xi\rt\in \partial^C Dist_\ccal A(z\rt)$ (where $\partial^C$ stands for the generalized subdifferential, cfr. \cite{aubin2009set,clarke1990optimization}), we have
\eee{
Dist_\ccal A(\hat x\rt)=Dist_\ccal A(\hat x(t_0))+\ps{\xi\rt}{\hat x\rt-\hat x(t_0)}.
}
We recall any vector of the form $v = (x-\tilde x)/dist(x,E)$, with $\tilde x$ in the projection set of $x\neq \tilde x$ onto $E$, is a proximal normal to $E$ at $\tilde x$, and so it lays in $(\ttnn{cl co } \scr N_E(x))$. Hence, from \cite[Theorem 8.49]{rockafellar2009variational} and the well known characterization of the subdifferential of the distance function in terms of projection sets (cfr.   \cite[pp. 340-341]{rockafellar2009variational}), we have that there exist $\lambda_i>0$, $y_i$ in the set of projections of $z\rt$ onto $\ttnn{bdr }\ccal A$, and $\xi_i\in \scr N^C_\ccal A(y_i)\cap \bb S^{n-1}$, for $i=1,...,N$, with $1\mineq N\mineq n+1$, such that $\sum_{i=1}^{N} \lambda_i=1$ and $\xi\rt=\sum_{i=1}^{N} \lambda_i\xi_i$. Moreover, for any $s\in[t_0,t_1]$
\eee{
\modulo{y_i-\hat x\rs}&\mineq \modulo{y_i-z\rt}+\modulo{z\rt-\hat x\rs}\\
&\mineq dist(z\rt,\ttnn{bdr }\ccal A)+\modulo{\hat x(t_0)-\hat x\rs}\\
&\mineq dist(\hat x(t_0),\ttnn{bdr }\ccal A)+\modulo{z\rt-\hat x(t_0)}+\modulo{\hat x(t_0)-\hat x\rs}\\
&\mineq \dfrac{\eta}{4}+2\dfrac{\eta}{8}=\eta.
}
Hence, $Dist_\ccal A(\hat x\rt)=Dist_\ccal A(\hat x(t_0))+\sum_{i=1}^N \lambda_i \int_{t_0}^{t}\ps{\xi_i}{\hat x'\rs}\,ds<0$ for all $t\in ]t_0,t_1]$, and $x\ccd=\hat x\ccd$ satisfies the conclusions.

\textsc{Case 2}: $\mu(\ccal A_+)>0$.

\noindent Applying the Measurable Selection Theorem (cfr. \cite{aubin2009set}), let $w:\ccal A_+\ra \bb R^n$ be the Lebesgue measurable function satisfying $w\rt \in Q(t,\hat x\rt)$ for a.e. $t\in\ccal A_+$, as in \textbf{Ax.C.2}. Define
\begin{equation*}
\tau:=
\begin{cases}
t_1 &\ttnn{if }\mu(\ccal A_+)\mineq k\rho\\
\min \graffe{t\in]t_0,t_1]\,|\, \mu(\ccal A_+ \cap [t_0,t])=k\rho}&\ttnn{otherwise},
\end{cases}
\end{equation*}
and keep $y:[t_0,t_1]\ra \bb R^n$ the arc satisfying $y(t_0)=\hat x(t_0)$ and
\begin{equation}\label{ydef'}
y'\rt:=
\begin{cases}
w\rt &\ttnn{if }t\in \ccal A_+\cap [t_0,\tau]\\
\hat x'(t)&\ttnn{if }t\in {[t_0,t_1]\backslash \tonde{\ccal A_+ \cap [t_0,\tau]}} \cap J.
\end{cases}\end{equation}
So, we have for all $t\in[t_0,t_1]$
\equazioneref{stima_y_1'}{
\norm{\hat x -y}_{\ccal W^{1,1},[t_0,t]}&=\modulo{y(t_0)-\hat x(t_0)}+\int_{t_0}^{t} \modulo{y'(s)-\hat x'(s)}\,ds\\
&=\int_{t_0}^{\tau \wedge t}\modulo{w(s)-\hat x'(s)}\chi_{\ccal A_+\cap [t_0,\tau]}(s)\,ds\mineq 2M\mu(\ccal A_+\cap [t_0,\tau \wedge t]).
}
Applying Filippov's Theorem (cfr. \cite{aubin2009set}), we have that there exists an $Q$-trajectory $x\ccd$ on $[t_0,t_1]$ such that $x(t_0)=y(t_0)$ and
\equazioneref{filippov'}{
\norm{y-x}_{\ccal W^{1,1},[t_0,t]}\mineq e^{\int_{t_0}^t \varphi (\tau)\,d\tau}{\int_{t_0}^t dist(y'(s),Q(s,y(s)))}\,ds
}
for all $t\in [t_0,t_1]$. Now, since for any $s\in[t_0,t_1]$, $y(s)=\hat x(t_0)+\int_{t_0}^{s}y'(s)\,ds=\hat x(s)+\int_{t_0}^{s}\tonde{y'(\xi)-\hat x'(\xi)}\,d\xi$, then for a.e. $s\in[t_0,t_1]$ we have $Q(s,\hat x(s))\subset Q(s,y(s))+\varphi(s)\int_{t_0}^{s}\modulo{y'(\xi)-\hat x'(\xi)}\,d\xi\, \bb B^n$. So, taking note of \rif{ydef'}, we have for a.e. $s\in[t_0,t_1]$
\equazioneref{stima_d_F'}{
dist(y'(s), Q(s,y(s)) )&\mineq dist(y'(s),Q(s,\hat x(s)))+\varphi(s)\int_{t_0}^{s}\modulo{y'(\xi)-\hat x'(\xi)}\,d\xi\\
&=\varphi(s)\int_{t_0}^{s}\modulo{y'(\xi)-\hat x'(\xi)}\,d\xi.
}
Using \rif{stima_d_F'} and \rif{stima_y_1'}, it follows that $dist(y'(s),Q(s,y(s)))\mineq 2\varphi(s)M\mu(\ccal A_+\cap [t_0,\tau \wedge s])$ for a.e. $s\in[t_0,t_1]$. Hence, we obtain the estimates for all $t\in[t_0,t_1]$
\eee{
\int_{t_0}^t dist(y'(s),Q(s,y(s)))\,ds\mineq
2\theta_\varphi(\Delta)M\mu(\ccal A_+\cap [t_0,\tau \wedge t]).
}
Thus, by \rif{filippov'}, we have $\norm{y-x}_{\ccal W^{1,1},[t_0,t]}\mineq 2e^{\theta_\varphi(\Delta)}{\theta_\varphi(\Delta)M}\mu(\ccal A_+\cap [t_0,\tau \wedge t])$ for all $t\in[t_0,t_1]$, and, using \rif{stima_y_1'}, we get
\eee{
\norm{\hat x-x}_{\ccal W^{1,1},[t_0,t_1]}\mineq \beta \rho\;\text{ with } \beta=2M\tonde{e^{\theta_\varphi(\Delta)}\theta_\varphi(\Delta)+1}k.
}

To conclude the proof, we show that
\[x\rt\in \ttnn{int}\,\ccal A\quad \forall t\in]t_0,t_1].\]

Consider $t\in]t_0,\tau]$. Then, applying the main value theorem, we have $Dist_\ccal A(y\rt)=Dist_\ccal A(y(t_0))+\ps{\xi\rt}{y\rt-y(t_0)}$ for some $z\rt\in[y(t_0),y(t)]$ and $\xi\rt\in \partial^C Dist_\ccal A(z\rt)$. Now, we have that there exist $\lambda_i>0$, $y_i$ in the set of projections of $z\rt$ onto $\ttnn{bdr }\ccal A$, and $\xi_i\in \scr N^C_\ccal A(y_i)\cap \bb S^{n-1}$, for $i=1,...,N$, with $1\mineq N\mineq n+1$, such that $\sum_{i=1}^{N} \lambda_i=1$ and $\xi\rt=\sum_{i=1}^{N} \lambda_i\xi_i$. Since for all $s\in [t_0,t_1]$ and all $i=1,...,N$, $\modulo{y_i-\hat x(s)}\mineq \modulo{y_i-z\rt}+\modulo{z\rt-\hat x(t_0)}+\modulo{\hat x(s)-\hat x(t_0)}$ and 
\eee{
\modulo{y_i-z\rt}&\mineq dist(y\rt,\ttnn{bdr }\ccal A)+\modulo{dist(z\rt,\ttnn{bdr }\ccal A)-dist(y\rt,\ttnn{bdr }\ccal A)}\\
&\mineq dist(y\rt,\ttnn{bdr }\ccal A)+\modulo{y\rt-y(t_0)}\\
&\mineq dist(y(t_0),\ttnn{bdr }\ccal A)+2\modulo{y\rt-y(t_0)},
}
we have that $\modulo{y_i-\hat x(s)}\mineq \dfrac{\eta}{4}+2\dfrac{\eta}{8}+2\dfrac{\eta}{8}<\eta$.
It follows that
\equazioneref{stima_xi_chiave'}{
\ps{\xi\rt}{y\rt-y(t_0)}&=\int_{[t_0,t]\backslash \ccal A_+} \ps{\xi\rt}{\hat x'\rs}\,ds+\int_{[t_0,t]\cap \ccal A_+} \ps{\xi\rt}{w\rs}\,ds\\
&=\sum_{i=1}^{N} \int_{[t_0,t] \backslash \ccal A_+} \ps{\xi_i}{\hat x'\rs}\,ds+\sum_{i=1}^{N} \int_{[t_0,t]\cap \ccal A_+} \lambda_i\ps{\xi_i}{w\rs}\,ds\\
&\mineq -r\mu(\ccal A_+\cap [t_0,t]).
}
Hence, using \rif{cond2_1}-(ii),
\equazioneref{stima_finale'}{
Dist_\ccal A(x\rt)&\mineq \modulo{x\rt-y\rt}+Dist_\ccal A(y\rt)\\
&\mineq \tonde{2e^{\theta_\varphi(\Delta)}{\theta_\varphi(\Delta)M}-r}\mu(\ccal A_+\cap [t_0, t])\mineq 0.
}
The equality in \rif{stima_xi_chiave'} occurs only if $\mu([t_0,t]\backslash \ccal A_+)=0$, and in that case it follows that $\mu(\ccal A_+\cap [t_0,t])$ has positive measure. It follows that the inequality in \rif{stima_finale'} is strict.

Consider $t\in]\tau,t_1]$. By the Main Value Theorem (cfr. \cite[Theorem 2.3.7]{clarke1990optimization}), for some $z\rt\in[\hat x\rt,y(t)]$ and $\xi\rt\in \partial^C Dist_\ccal A(z\rt)$, we have $Dist_\ccal A(y\rt)=Dist_\ccal A(\hat x\rt)+\ps{\xi\rt}{y\rt-\hat x\rt}$.
Furthermore, arguing as in the previous case, consider $\lambda_i>0$, $y_i$ in the set of projections of $z\rt$ onto $\ttnn{bdr }\ccal A$, and $\xi_i\in \scr N^C_\ccal A(y_i)\cap \bb S^{n-1}$, for $i=1,...,N$, with $1\mineq N\mineq n+1$, such that $\sum_{i=1}^{N} \lambda_i=1$ and $\xi\rt=\sum_{i=1}^{N} \lambda_i\xi_i$.
We notice that for all $i=1,...,N$ and $s\in ]\tau,t_1]$, $\modulo{y_i-\hat x(s)}\mineq \modulo{y_i-z\rt}+\modulo{z\rt-\hat x\rs}$
and
\eee{
&\modulo{z\rt-\hat x\rs}\\
&\mineq \dfrac{1}{2}(\modulo{z\rt-y\rt}+\modulo{y\rt-y(t_0)}+\modulo{\hat x(t_0)-\hat x\rs}+\modulo{\hat x\rs-\hat x\rt}\\
&\qquad \qquad \qquad +\modulo{\hat x\rt-z\rt})\\
&\mineq \modulo{\hat x\rt-y\rt}+\dfrac{1}{2}\tonde{\modulo{y\rt-y(t_0)}+\modulo{\hat x(t_0)-\hat x\rs}+\modulo{\hat x\rs-\hat x\rt}}.
}
Now, since $z\rt=ay\rt+(1-a)\hat x\rt$, where $a\in[0,1]$, we have that $\modulo{y_i-z\rt}\mineq Dist_\ccal A(\hat x(t_0))+\modulo{\hat x(t_0)-z\rt}\mineq \dfrac{\eta}{4}+\modulo{y\rt-y(t_0)}+\modulo{\hat x\rt-\hat x(t_0)}.$
Summing up we obtain $\modulo{y_i-\hat x(s)}\mineq \tonde{\dfrac{\eta}{4}+2\dfrac{\eta}{8}}+2\dfrac{\eta}{8}+\dfrac{1}{2}3\dfrac{\eta}{8}<\eta$.
Then
\eee{
\ps{\xi\rt}{y\rt-\hat x\rt}&=\int_{t_0}^t \ps{\xi\rt}{ y'\rs-x'\rs}\,ds\\
&=\int_{\ccal A_+\cap [t_0,\tau]} \ps{\xi\rt}{ w\rs-x'\rs}\,ds\\
&=\sum_{i=1}^{N} \int_{\ccal A_+\cap [t_0,\tau]} \lambda_i\ps{\xi_i}{ w\rs-x'\rs}\,ds\\
&\mineq -r\mu(\ccal A_+\cap [t_0,\tau]).
}
Finally, by \rif{cond2_1}-(iii),
\eee{
Dist_\ccal A(x\rt)&\mineq \modulo{x\rt-y\rt}+Dist_\ccal A(y\rt)\\
&\mineq 2e^{\theta_\varphi(\Delta)}{\theta_\varphi(\Delta)Mk \rho}+Dist_\ccal A(y\rt)\\
&\mineq Dist_\ccal A(\hat x\rt)+\tonde{2e^{\theta_\varphi(\Delta)}{\theta_\varphi(\Delta)M}-r}k \rho\\
&\mineq (1+(2e^{\theta_\varphi(\Delta)}{\theta_\varphi(\Delta)M}-r)k) \rho<0.
}
\\
\end{proof}

\bibliographystyle{plain}
\bibliography{BIBLIO_VBHF_comparison}

\end{document}